\documentclass{scrartcl}

\usepackage[english]{babel}
\usepackage[utf8]{inputenc}
\usepackage{mathtools}
\usepackage{mathpazo}
\usepackage{amsmath}
\usepackage{amssymb}
\usepackage{amsthm}
\usepackage{graphicx}
\usepackage{subcaption}
\usepackage{aliascnt}
\usepackage[top=1.25in]{geometry}
\usepackage{enumitem}
\usepackage{todonotes}
\usepackage[normalem]{ulem}
\setenumerate{label=\textup{(\alph*)}}
\usepackage[hyphens]{url}
\usepackage[hidelinks]{hyperref}
\usepackage[capitalise]{cleveref}
\usepackage{mathrsfs}

\usepackage{pgfplots}
\pgfplotsset{compat=1.18}
\definecolor{tab-blue}{HTML}{1F77B4}
\definecolor{tab-orange}{HTML}{FF7F0E}
\definecolor{tab-brown}{HTML}{8C564B}

\newtheorem{theorem}{Theorem}[section]

\newaliascnt{lemma}{theorem}
\newtheorem{lemma}[lemma]{Lemma}
\aliascntresetthe{lemma}

\newaliascnt{proposition}{theorem}
\newtheorem{proposition}[proposition]{Proposition}
\aliascntresetthe{proposition}

\newaliascnt{corollary}{theorem}
\newtheorem{corollary}[corollary]{Corollary}
\aliascntresetthe{corollary}

\theoremstyle{definition}

\newaliascnt{remark}{theorem}
\newtheorem{remark}[remark]{Remark}
\aliascntresetthe{remark}

\newaliascnt{definition}{theorem}

\aliascntresetthe{definition}

\newaliascnt{example}{theorem}
\newtheorem{example}[example]{Example}
\aliascntresetthe{example}

\newtheorem{assumption}{Assumption}

\crefdefaultlabelformat{#2\textup{#1}#3}
\Crefname{assumption}{Assumption}{Assumptions}
\crefname{assumption}{Assumption}{Assumptions}
\Crefrangeformat{assumption}{Assumptions #3\textup{#1}#4--#5\textup{#2}#6}
\crefname{proposition}{Proposition}{Propositions}
\Crefname{proposition}{Proposition}{Propositions}
\crefname{lemma}{Lemma}{Lemmas}
\Crefname{lemma}{Lemma}{Lemmas}

\newcommand{\keywords}[1]
{
	{\small
		\textbf{Key words.} {#1}
		\\
	}
}

\renewcommand{\abstract}[1]
{
	{\small
		\textbf{Abstract.} {#1}
		\\
	}
}

\begin{document}

\title{Asymptotic Analysis of Empirical Dynamic Programming in Infinite-Horizon Stochastic Optimal Control\thanks{\textbf{Funding}:
The research  was supported by the National Science Foundation under Grant No.\ DMS-2410944.
}}
\author{Xin Chen\footnotemark[2] \and Elif Sena Isik\footnotemark[2] \and Johannes Milz\thanks{H.\ Milton Stewart School of Industrial and Systems Engineering, Georgia Institute of Technology, Atlanta, Georgia
30332, USA (\texttt{xin.chen@isye.gatech.edu}, 
\texttt{eisik6@gatech.edu}, \texttt{johannes.milz@isye.gatech.edu})}
}

\maketitle

\abstract{%
We derive statistical limit theorems for sample-based approximations of infinite-horizon discounted stochastic optimal control problems in discrete time. Our first result is a functional central limit theorem for the sample-based value function under a uniqueness-type condition on population optimal policies. The limiting law is a mean-zero Gaussian process characterized by a linear fixed point equation that resembles a dynamic programming principle. We compare these asymptotics with those obtained from sample-based policy optimization and illustrate that their limiting variances can be different. We also derive a limit theorem for models with  nonunique optimal policies, where the limiting law may be non-Gaussian. Applications to inventory control and renewable harvesting illustrate the theory.
}

\keywords{empirical dynamic programming; stochastic dynamic programming;
random fixed-point equations; sample average approximation; central
limit theorem; Z-estimation}

\section{Introduction}

Empirical dynamic programming (DP) is a basic computational approach in
data-driven infinite-horizon stochastic optimal control (SOC). When
samples from the distribution of the exogenous noise are available, a
natural approach is to replace the expectation in the DP equation by a
sample average and solve the resulting sample average approximation
(SAA) version of the DP problem \cite{Kleywegt2002,cltshapcheng21}. The output is the SAA value function, a random function determined by the sample. Its
asymptotic law provides insights on how sampling error enters 
and propagates through the DP fixed point equation. 
The resulting law provides a large-sample distributional approximation for the SAA value function and thereby supports uncertainty quantification.

This paper develops statistical limit theory for this SAA value
function. The main object is the error between the SAA value function and
the population value function. 
Under regularity conditions, including a uniqueness-type condition on the one-step outcomes induced by population optimal policies and an empirical linearization condition, 
we characterize the 
error's limiting law under a uniqueness-type condition on the one-step outcomes induced
by population optimal policies, compare empirical DP with
trajectory-based SAA policy optimization, and extend the analysis to
models with nonunique optimal policies.

The infinite-horizon setting differs from both static stochastic
optimization and finite-horizon DP\@. In static SAA, sampling error enters
a single random optimization problem, so classical stochastic
programming asymptotics apply \cite{SDR}. In
finite-horizon DP, the error is propagated through a finite backward
recursion, leading to stagewise recursive limits \cite{Milz2025}. In
discounted infinite-horizon DP, by contrast, the SAA value function is
the fixed point of the SAA DP operator. Its sampling error is
therefore governed by a random fixed point equation.

To separate ordinary sampling error from the additional error created by
the fixed-point equation, we first evaluate the SAA DP operator at
the population value function. Let \(V\) denote the population value
function and let \(\mathcal T\) denote the corresponding DP operator.
The population value function satisfies the fixed point equation
$
V=\mathcal T V
$.
Given $N$ independent and identically distributed (i.i.d.) observations of
the model's random noise, let \(\hat{\mathcal T}_N\) denote the DP
operator obtained by replacing the true noise distribution with the
empirical distribution of the sample.
Because \(V\) is deterministic,
\[
N^{1/2}(\hat{\mathcal T}_N V-\mathcal T V)
\]
is the first-order empirical error of a state-indexed family of SAA optimization problems. No random fixed point is involved in this term. Its analysis is therefore analogous to static stochastic optimization and it yields the functional central limit theorem
(CLT)
\[
N^{1/2}(\hat{\mathcal T}_N V-\mathcal T V)
\rightsquigarrow
\mathfrak H,
\]
where \(\mathfrak H\) is a Gaussian process. 

To propagate this static sampling error through the SAA fixed point equation, recall that \(V=\mathcal T V\), while
the SAA value function satisfies
$
\hat V_N=\hat{\mathcal T}_N\hat V_N
$.
Let \(\gamma\in(0,1)\) denote the discount factor,
and let \(\mathcal L\) be the closed-loop
transition operator under a population optimal policy.
This operator can be thought of as a Hadamard
linearization of the DP operator $\mathcal{T}$,
and its operator norm is at most one.
This linearization is justified under the uniqueness-type
condition on the population optimal policies mentioned above.
Using the fixed point equations, we have
\[
\begin{aligned}
N^{1/2}(\hat V_N-V)
&=
N^{1/2}\big(\hat{\mathcal T}_N V-\mathcal T V
+
\gamma \mathcal L (\hat V_N-V) \big) \\
&\quad+
N^{1/2}
\bigl(
\hat{\mathcal T}_N\hat V_N
-
\hat{\mathcal T}_N V
-
\gamma\mathcal L(\hat V_N-V)
\bigr).
\end{aligned}
\]
If the empirical linearization condition
\[
N^{1/2}
\bigl(
\hat{\mathcal T}_N\hat V_N
-
\hat{\mathcal T}_N V
-
\gamma\mathcal L(\hat V_N-V)
\bigr)
=
o_p(1)
\]
is satisfied,
then
\[
(I-\gamma\mathcal L)\big[N^{1/2}(\hat V_N-V)\big]
=
N^{1/2}(\hat{\mathcal T}_N V-\mathcal T V)+o_p(1).
\]
Consequently,
\[
N^{1/2}(\hat V_N-V)
\rightsquigarrow
(I-\gamma\mathcal L)^{-1}\mathfrak H.
\]
Thus the CLT has two ingredients: a static SAA limit
\(\mathfrak H\) and its propagation through the optimal closed-loop
dynamics.

The main technical challenge is verifying the empirical linearization condition.
This condition asserts that, at the \(N^{-1/2}\) scale, replacing the
deterministic argument \(V\) of the empirical DP operator by the
SAA value function \(\hat V_N\) contributes only the linear term
\(\gamma\mathcal L(\hat V_N-V)\). We verify this condition by separating
it into two parts: a local linearization of the population DP
operator and a stochastic equicontinuity-type condition. Our sufficient conditions,
including compactness of a transition operator, finite state-action
models, and  Donsker-type conditions, are designed to
make this empirical linearization checkable in SOC models.

We compare empirical DP with an idealized trajectory-based SAA formulation that samples independent disturbance trajectories and minimizes the average discounted cost over a parameterized policy class.
Under a unique population optimal policy and a well-specified
policy class, its limiting variance equals the variance of the realized
total discounted cost under that policy.  A finite state-action example demonstrates
that this variance may differ from the limiting SAA DP optimal value variance.

We also consider SOC models in which population optimal policies can
induce different one-step costs or transitions. In this 
nonunique case, the limiting law may be non-Gaussian. We derive
a corresponding statistical limit theorem.

\subsection{Contributions}

The paper makes three main contributions.
\begin{enumerate}[label=\textup{(\roman*)},nosep]

\item
We establish a functional CLT for SAA value functions
under uniqueness of optimal one-step costs and transitions. The Gaussian
limit solves a linear random fixed point equation, and we provide
verifiable sufficient conditions for the required linearization and
stochastic equicontinuity conditions.

\item
We compare the SAA DP formulation with trajectory-based SAA policy
optimization. Under a unique population optimal policy, 
the latter has limiting variance equal to the variance of the
total discounted cost. 
We compare this variance with the limiting variance
of the SAA DP optimal value.

\item
We apply our theory in inventory control and renewable
harvesting models.
\end{enumerate}

For completeness, \Cref{sec:nonunique-policy-limit} gives a supplementary
functional statistical limit theorem beyond uniqueness of optimal
one-step costs and transitions. The resulting limit solves a nonlinear
random fixed point equation and need not be Gaussian.

\subsection{Related literature}

\paragraph{Static stochastic optimization.}
CLTs for SAA optimal values are standard in static stochastic
optimization \cite{Shapiro2003,SDR}. Although our models are dynamic,
parts of the analysis use ideas from this literature. 
Statistical limit theorems for
two-stage stochastic programs are given in \cite{Eichhorn2007}.

\paragraph{Empirical dynamic programming.}
Sample-based approximations of dynamic programs are studied by \cite{Haskell2016,Haskell2020} and related work on iterated random operators \cite{Gupta2024}. 
The finite-horizon CLT theory of \cite{Milz2025} is closest to the
present work. There, SAA value functions solve stochastic backward recursions, which
give a recursive characterization of the limiting laws. Since we consider discounted
infinite-horizon models, the SAA value function is instead the fixed point of the
SAA DP operator. The papers most closely related to the finite state-action
setting of our work are \cite{Mannor2007,Zhu2024}. 
For finite state-action discounted reward Markov decision processes,
\cite{Mannor2007} derives approximations to the bias and variance 
of estimated value functions. For a similar problem
class, \cite{Zhu2024} establishes CLTs for estimated
action-value and value functions. Recent finite-state inference results derive
CLTs for value and action-value functions 
under fixed policies \cite{Su2025}.

\paragraph{Pathwise performance limits.}
A separate line of work studies statistical limits for rewards or
costs realized along sample paths of dynamic programs; see, for example,
\cite{Arlotto2016}. These results concern the law of a realized
performance criterion when the transition law, cost functions, and
decision rule are fixed. In the present paper, the main 
random object is instead the SAA value function.

\paragraph{Inventory control.}
For inventory control problems, \cite{cltshapcheng21} establishes CLTs for
SAA optimal values. Our analysis instead develops a functional CLT for the
SAA value function, from which a CLT for the SAA optimal value follows by
evaluation at the initial state. 
Asymptotic normality for the estimated parameters
of data-driven \((s,S)\) policies and for optimal-cost estimators
is derived in \cite{Zhang2025}. In the present work, we do not
study the asymptotic distribution of SAA policies.

\paragraph{Z-estimation.}
The population and SAA DP equations are operator equations and
therefore fit naturally into the framework of Z-estimation in function spaces \cite{Vaart2023}. However, Theorem 3.3.1 in \cite{Vaart2023} does not by itself provide a CLT for our setting. Its hypotheses require a 
tightness condition on the scaled difference of population and SAA value
functions, and this property may not be 
easily verified for SOC models.

\subsection{Outline}

\Cref{sec-basic} develops statistical limit theory for SAA optimal values. \Cref{sec:policy-saa} derives a trajectory-based
policy-optimization CLT, 
computes closed-loop variances, and 
discusses a finite state-action example.
\Cref{sec:inventory-control,sec:renewable-resource-harvesting} discuss applications.

\section{Infinite-horizon stochastic optimal control in discrete time}
\label{sec-basic}

We consider an infinite-horizon discounted SOC problem with discount
factor \(\gamma\in(0,1)\). For a deterministic initial state
\(x_0\in\mathcal X\), the problem is given by
\begin{equation}
\label{eq:inf-soc}
\inf_{\pi\in\Pi}
\mathbb{E}
\bigg[
\sum_{t\geq 0}
\gamma^t f(x_t,\pi(x_t),\xi_t)
\bigg]
\quad
\text{s.t.}
\quad
x_{t+1}=F(x_t,\pi(x_t),\xi_t),
\quad t=0,1,2,\ldots .
\end{equation}
Here \(\Pi\) denotes the class of measurable
stationary Markov policies
\(\pi:\mathcal X\to\mathcal U\). The state space is
\(\mathcal X\subset\mathbb R^n\), the control space is
\(\mathcal U\subset\mathbb R^m\), and \((\xi_t)\) is an i.i.d.\
sequence with common distribution \(P\). The distribution \(P\) is
time-homogeneous 
 with closed support
\(\Xi\subset\mathbb R^d\), and does not depend on the state or control. The
function \(f:\mathcal X\times\mathcal U\times\Xi\to\mathbb R\) is the
stage cost, and \(F:\mathcal X\times\mathcal U\times\Xi\to\mathcal X\)
is the state-update map. 

\subsection{Model assumptions and DP operators}

We impose basic regularity conditions on our SOC model.

\begin{assumption}[{Compact Lipschitz  model}]
\label{inf-soc:ass-2}
\textup{(i)}
The sets $\mathcal{X}$, $\mathcal{U}$, and $\Xi$ are compact.
\textup{(ii)}~The maps $f:\mathcal{X}\times\mathcal{U}\times\Xi\to \mathbb{R}$
and $F:\mathcal{X}\times\mathcal{U}\times\Xi\to \mathcal{X}$  are continuous.
\textup{(iii)}~There exists a square integrable function  $K \colon \Xi  \to [0,\infty)$
such that 
for all   $\xi\in \Xi$,
$f(\cdot, \cdot,\xi)$ and 
$F(\cdot, \cdot,\xi)$ are Lipschitz
continuous with Lipschitz constant $K(\xi)$.
\end{assumption}

We now define the action-value operators
under \Cref{inf-soc:ass-2}.
For a compact set \(E\), let \(C(E)\) denote the space of continuous
real-valued functions on \(E\), equipped with the sup norm
\(\|\cdot\|_\infty\).  
We define the
action-value operator
\(\mathcal Q:C(\mathcal X)\to C(\mathcal X\times\mathcal U)\) by
\begin{align}
\label{eq:Q}
[\mathcal Q(W)](x,u)
\coloneqq 
\mathbb{E}_{\xi\sim P}
\big[
f(x,u,\xi)+\gamma W(F(x,u,\xi))
\big]
\end{align}
and its SAA counterpart \(\hat{\mathcal{Q}}_N:C(\mathcal X)\to C(\mathcal X\times\mathcal U)\) by
\begin{align}
\label{eq:hatQN}
[\hat{\mathcal{Q}}_N(W)](x,u)
\coloneqq 
\frac{1}{N}
\sum_{i=1}^N
\big[
f(x,u,\xi^{(i)})+\gamma W(F(x,u,\xi^{(i)}))
\big],
\end{align}
where $\xi^{(1)}, \xi^{(2)}, \ldots $ are independent and each $\xi^{(i)}$ has the same
distribution as that of $\xi \sim P$.
We define the DP operator
\(\mathcal{T} \colon C(\mathcal{X}) \to  C(\mathcal{X})\)
and the SAA DP operator
\(\hat{ \mathcal{T}}_{N} \colon C(\mathcal{X}) \to  C(\mathcal{X})\) by
\begin{equation}\label{eq:t-operators}
[\mathcal{T} W](x) \coloneqq  \min_{u\in \mathcal{U}}  \, 
[\mathcal Q(W)](x,u),
\quad \text{and} \quad 
[\hat{\mathcal{T}}_{N} W](x) \coloneqq  \min_{u\in \mathcal{U}}  \, 
[\hat{\mathcal{Q}}_N(W)](x,u).
\end{equation}

The population and empirical value functions are the unique fixed points
\begin{equation}
\label{eq:bellman-fixed-points}
V=\mathcal{T}V,
\quad \text{and} \quad 
\hat V_N=\hat{\mathcal{T}}_N\hat V_N.
\end{equation}
We define the population optimal solution set
\begin{equation}\label{eq:solutionset}
\mathcal{U}^\star(x)
\coloneqq \operatorname*{argmin}_{u\in \mathcal{U}}  \, 
[\mathcal Q(V)](x,u).
\end{equation}

Under \Cref{inf-soc:ass-2}, $\mathcal T$ and $\hat{\mathcal{T}}_{N}$ are $\gamma$-contractions on $C(\mathcal X)$, the minima in \eqref{eq:t-operators} 
are attained, and the fixed points in \eqref{eq:bellman-fixed-points} exist uniquely. 
Moreover, there exists an optimal policy
\(\pi^\star\) for the population problem.

The next assumption requires all population optimal actions at a given
state to induce the same one-step cost and next-state transition for
every disturbance realization.

\begin{assumption}[Uniqueness of optimal one-step costs and transitions]
\label{ass:uniqueness-type}
There exist continuous maps
\(f^\star:\mathcal X\times\Xi\to\mathbb R\) and
\(F^\star:\mathcal X\times\Xi\to\mathcal X\) such that, for every
\(x\in\mathcal X\), every \(u\in\mathcal{U}^\star(x)\), and every
\(\xi\in\Xi\),
\[
f(x,u,\xi)=f^\star(x,\xi),
\quad \text{and} \quad 
F(x,u,\xi)=F^\star(x,\xi).
\]
\end{assumption}

\Cref{ass:uniqueness-type} concerns uniqueness of the outcome induced
by an optimal control rather than uniqueness of the policy itself.
Multiple population optimal controls may exist at a given state, as
long as they induce the same one-period cost and next state for every
noise realization. If problem \eqref{eq:inf-soc} has
a unique optimal policy \(\pi^\star\), 
then  \(\pi^\star\) is continuous (see Corollary~8.1 in \cite{Hogan1973}), and
\Cref{ass:uniqueness-type} holds with
$
f^\star(x,\xi)=f(x,\pi^\star(x),\xi)
$
and 
$F^\star(x,\xi)=F(x,\pi^\star(x),\xi)$.
Thus, uniqueness of the population optimal control is sufficient for
\Cref{ass:uniqueness-type}.

Under \Cref{ass:uniqueness-type}, 
we define the operator
\(\mathcal L:C(\mathcal X)\to C(\mathcal X)\) 
by
\begin{align}
\label{eq:L}
[\mathcal L W](x)
=
\mathbb{E}_{\xi\sim P}
\big[
W(F^\star(x,\xi))
\big],
\qquad W\in C(\mathcal X).
\end{align}
The operator norm of \(\mathcal L\) is
at most \(1\). Since \(\gamma\in(0,1)\), the operator
\(I-\gamma\mathcal L\) is continuously invertible, where
\(I:C(\mathcal X)\to C(\mathcal X)\) is the identity map.
We define
\begin{align}
\label{eq:Phi-star}
\Phi^\star(x,\xi)
\coloneqq
f^\star(x,\xi)+\gamma V(F^\star(x,\xi)).
\end{align}

\subsection{Functional CLT}
\label{subsect:inf-soc:clt}

This section states the fixed point CLT for the SAA value function. 
Proofs are deferred to \Cref{app:main-clt-proofs}.
The functional CLT  is expressed in terms of a centered Gaussian
process \(Z=\{Z(x):x\in\mathcal X\}\), meaning that for any finite collection of
states \(x_1,\ldots,x_k \in \mathcal{X}\), the random vector
\((Z(x_1),\ldots,Z(x_k))\) has a mean-zero multivariate normal distribution.
We denote
convergence in distribution of random elements by
\(\rightsquigarrow\).

\begin{theorem}
\label{thm:inf-soc:clt}
Suppose that  \Cref{inf-soc:ass-2,ass:uniqueness-type} hold. Then 
\begin{align}
\label{eq:inf-soc:clt-mathcalV}
N^{1/2}
(\hat{\mathcal{T}}_{N}V - \mathcal{T}V)
\rightsquigarrow
\mathfrak{H}
\quad \text{in} \quad C(\mathcal{X}),
\end{align}
where $\mathfrak{H}$ is a centered
Gaussian process on $\mathcal{X}$ 
with covariance function
\begin{align}
\label{covHt}
(x,x')
\mapsto
\operatorname{Cov}_{\xi\sim P}
\big(
\Phi^\star(x,\xi),
\Phi^\star(x',\xi)
\big).
\end{align}
If, additionally,
\begin{align}
\label{eq:inf-soc:fixed-point-linearization}
N^{1/2}
\bigl(
\hat{\mathcal T}_N\hat V_N
-
\hat{\mathcal T}_N V
-
\gamma\mathcal L(\hat V_N-V)
\bigr)
=
o_p(1)
\quad\text{in}
\quad 
C(\mathcal X),
\end{align}
then
\begin{align}
\label{eq:inf-soc:limit}
N^{1/2}(\hat{V}_{N} - V) \rightsquigarrow 
\mathfrak{G} \coloneqq  (I-\gamma \mathcal{L})^{-1} \mathfrak{H}
\quad \text{in} \quad C(\mathcal{X}).
\end{align}
\end{theorem}

The functional CLT \eqref{eq:inf-soc:limit} also yields a pointwise CLT\@.
Indeed, for every fixed \(x_0\in\mathcal X\), the continuous
mapping theorem applied to the evaluation map \(W\mapsto W(x_0)\)
gives
\begin{align}
\label{eq:inf-soc:pointwise-clt}
N^{1/2}\bigl(\hat V_N(x_0)-V(x_0)\bigr)
\rightsquigarrow
\mathcal N\big(
0,\sigma_{\mathrm{DP},\gamma}^2(x_0)
\big),
\end{align}
where
\begin{align}
\label{eq:DP-variance}
\sigma_{\mathrm{DP},\gamma}^2(x_0)
\coloneqq
\operatorname{Var}\bigl(\mathfrak G(x_0)\bigr).
\end{align}
In the definition of the limiting variance, 
we make the dependence on the discount
factor $\gamma$ explicit for later use.

\begin{remark}[Finite state representation]
\label{rem:finite-state-bellman}
Let \Cref{inf-soc:ass-2,ass:uniqueness-type} hold, and suppose that
\(\mathcal X=\{x_1,\ldots,x_s\}\). We define
\[
P^\star(x,y)
\coloneqq
\operatorname{Prob}\{F^\star(x,\xi)=y\},
\qquad x,y\in\mathcal X.
\]
We identify \(C(\mathcal X)\) with \(\mathbb R^s\) by identifying each function
\(h\in C(\mathcal X)\) with the vector
\(\bigl(h(x_1),\ldots,h(x_s)\bigr)^\top\). Then,
\eqref{eq:L} becomes
\[
[\mathcal LW](x)
=
\mathbb{E}_{\xi\sim P}
\big[
W(F^\star(x,\xi))
\big]
=
\sum_{y\in\mathcal X}P^\star(x,y)W(y),
\qquad x\in\mathcal X.
\]
Thus, the matrix representing \(\mathcal L\) 
is given by \(P^\star\), and
\eqref{eq:inf-soc:limit} becomes
\[
N^{1/2}(\hat V_N-V)
\rightsquigarrow
(I-\gamma P^\star)^{-1}\mathfrak H
\qquad\text{in} \quad \mathbb R^s,
\]
where \(\mathfrak H\) is a centered Gaussian vector in \(\mathbb R^s\),
with covariance matrix given in \eqref{covHt}.
This representation parallels Corollary~1 of
\cite{Zhu2024}, which derives a CLT for the
estimated optimal value function of a finite state-action
discounted Markov decision process.
\end{remark}

\Cref{thm:inf-soc:clt} separates the analysis into a ``statistical step'' and a ``dynamical step.''  The statistical step is the one-period functional CLT \eqref{eq:inf-soc:clt-mathcalV}.  It is obtained by studying the SAA action-value map 
evaluated at the population value function and then differentiating the minimum operator.  No random fixed point appears in that part of the argument.  The dynamical step propagates this one-period perturbation through the fixed point equation.

The nontrivial issue is that the empirical operator in the fixed point equation is evaluated at the random argument \(\hat V_N\), not at \(V\).  The decomposition
\begin{align*}
\hat{\mathcal{T}}_N\hat V_N-\hat{\mathcal{T}}_NV-\gamma \mathcal{L}(\hat V_N-V)
& = \big\{\mathcal{T} \hat{V}_{N} - \mathcal{T} V
- \gamma \mathcal{L}(\hat{V}_{N}-V)\big\}
\\
& \quad +
 \big\{
\big[\hat{\mathcal{T}}_{N}
-
\mathcal{T}\big](\hat{V}_{N})
-
\big[
\hat{\mathcal{T}}_{N}- \mathcal{T}
\big]
(V)
 \big\}
\end{align*}
shows that it is enough to verify the two remainders
\begin{align}
N^{1/2}\bigl(\mathcal{T} \hat{V}_{N} - \mathcal{T} V
- \gamma \mathcal{L}(\hat{V}_{N}-V)\bigr)
&= o_p(1) \quad \text{in} \quad C(\mathcal{X}), \label{eq:inf-soc:linearization-error} \\
N^{1/2}
\big(\big[\hat{\mathcal{T}}_{N}
-
\mathcal{T}\big](\hat{V}_{N})
-
\big[
\hat{\mathcal{T}}_{N}- \mathcal{T}
\big]
(V)
\big)
&= o_p(1) \quad \text{in} \quad C(\mathcal{X}). \label{eq:inf-soc:continuity}
\end{align}
The first condition is a local linearization
condition for the population DP operator  around \(V\).  The second is a
stochastic equicontinuity-type condition that controls the SAA 
value function 
perturbation when its argument is changed from the deterministic
function \(V\) to the random fixed point \(\hat V_N\).  
\Cref{subsect:verifiable-conditions-fixed-point-clt} gives sufficient
conditions for \eqref{eq:inf-soc:linearization-error} and
\eqref{eq:inf-soc:continuity}.
The linearization condition
\eqref{eq:inf-soc:linearization-error} is necessary for this type
of CLT: If \(N^{1/2}(\hat V_N-V)\) converges in distribution in
\(C(\mathcal X)\), then the delta method applied to \(\mathcal T\) at
\(V\), with Hadamard  derivative \(\mathcal T'(V;W)=\gamma\mathcal LW\)
(see \Cref{lem:infimum-map-derivative}),
and a standard subsequence argument imply
\eqref{eq:inf-soc:linearization-error}.
We refer the reader to \Cref{app:main-clt-proofs}
for the definition of Hadamard 
differentiability. 

\subsection{Verifiable sufficient conditions}
\label{subsect:verifiable-conditions-fixed-point-clt}

In this section, we record the
main verifiable sufficient conditions for the linearization
condition \eqref{eq:inf-soc:linearization-error} and
the stochastic equicontinuity-type assumption \eqref{eq:inf-soc:continuity}; 
the proofs are given in
\Cref{subsect:proofs-sufficient-conditions}.

We first give sufficient conditions for the linearization condition
\eqref{eq:inf-soc:linearization-error}. One condition uses compactness
of an operator associated with the controlled dynamics.
Recall that a linear operator between Banach spaces is compact if the image of every bounded sequence has a convergent subsequence.
Under \Cref{inf-soc:ass-2}, 
we define
$\mathcal{A} \colon C(\mathcal{X}) \to 
C(\mathcal{X} \times \mathcal{U})$ by
\begin{align}
\label{eq:A-operator}
[\mathcal{A}W](x, u)
\coloneqq 
\mathbb{E}_{\xi\sim P}
[
W(F(x, u, \xi))
].
\end{align}
Thus \([\mathcal A W](x,u)\) is the expected value of \(W\) evaluated
at the next state.

We now collect three conditions sufficient for 
\eqref{eq:inf-soc:linearization-error}. After formulating our result, 
we discuss these conditions and outline parts of our proof strategy.

\begin{proposition}[Sufficient conditions for DP linearization]
\label{prop:main-linearization-conditions}
Suppose that \Cref{inf-soc:ass-2} holds. Then
\eqref{eq:inf-soc:linearization-error} holds if any of the following
conditions is satisfied:
\begin{enumerate}[label=\textup{(\roman*)},nosep]
\item \Cref{ass:uniqueness-type} holds, and 
\begin{align}
\label{eq:frechet-linearization-error}
\mathcal T\hat V_N-\mathcal T V
-\gamma\mathcal L(\hat V_N-V)
=
o_p(\|\hat V_N-V\|_\infty)
\quad\text{in} \quad  C(\mathcal X);
\end{align}
\item \Cref{ass:uniqueness-type} holds, and the transition operator
\(\mathcal A\) in \eqref{eq:A-operator} is compact;
\item The control space \(\mathcal U\) is finite, 
and the SOC problem \eqref{eq:inf-soc} has a unique optimal policy
\(\pi^\star\).
\end{enumerate}
\end{proposition}

The local expansion condition \eqref{eq:frechet-linearization-error}
may be viewed as a Fr\'echet differentiability-type condition. 
The compactness condition in
Condition~\textup{(ii)} relies instead on a smoothing property of the
controlled dynamics. The next result shows that this property holds in
a common class of additive-noise models. 
Condition~\textup{(iii)} uses stability of the optimal control rather
than smoothing. 
When \(\mathcal U\) is finite and the population optimal policy is
unique, small perturbations of the value function $V$
do not change the minimizing control, so the DP operator is
locally linear around \(V\).

\begin{lemma}[Additive noise smoothing]
\label{lemma:additive-noise-transition-linearization}
Let \(\mathcal X\) and \(\mathcal U\) be compact, let
\(G:\mathcal X\times\mathcal U\to\mathbb R^n\) and
\(\kappa: G(\mathcal{X} \times \mathcal{U}) +\Xi \to\mathcal X\) be continuous. Suppose that
\(F(x,u,\xi)=\kappa(G(x,u)+\xi)\)
and that  \(\xi\) has a density
with respect to Lebesgue measure.
Then \(\mathcal A\) is compact.
\end{lemma}

In \Cref{lemma:additive-noise-transition-linearization}, the density
of the noise gives a smoothing effect in \eqref{eq:A-operator}. The
operator \(\mathcal A\) becomes an integral-type operator, and
compactness follows from the Arzel\`a--Ascoli theorem. Additive noise
dynamics arise in inventory control.

We next give sufficient conditions for the stochastic equicontinuity
condition \eqref{eq:inf-soc:continuity}. 

\begin{proposition}[Sufficient conditions for stochastic equicontinuity]
\label{prop:main-equicontinuity-conditions}
Suppose that \Cref{inf-soc:ass-2,ass:uniqueness-type} hold.  Then
\eqref{eq:inf-soc:continuity} holds in either of the following cases:
\begin{enumerate}[label=\textup{(\roman*)},nosep]
\item the operator \(\mathcal A\) is compact and
\begin{align}
\label{eq:inf-soc:q-continuity}
N^{1/2}
\left(
[\hat{\mathcal Q}_N-\mathcal Q](\hat V_N)
-
[\hat{\mathcal Q}_N-\mathcal Q](V)
\right)
=o_p(1)
\quad\text{in}\quad C(\mathcal X\times\mathcal U);
\end{align}
\item 
the SOC problem \eqref{eq:inf-soc}
has a unique optimal policy $\pi^\star$, and
\(\mathcal X\) and \(\mathcal U\) are finite.
\end{enumerate}
\end{proposition}

The condition \eqref{eq:inf-soc:q-continuity} is a uniform version of
\eqref{eq:inf-soc:continuity} before minimization over controls is applied. 
For continuous models, the condition in \eqref{eq:inf-soc:q-continuity}
can be verified using empirical process results for classes indexed by
random functions; see
Section~3.13 in \cite{Vaart2023} and \cite{Vaart2007}. In our case,  the
random index is the SAA value function \(\hat V_N\). 

\begin{lemma}[Donsker conditions]
\label{lem:Lipschitz-equicontinuity}
Suppose that \Cref{inf-soc:ass-2,ass:uniqueness-type} hold
and that there exists
a closed set \(\mathcal{V} \subset C(\mathcal{X})\) 
of uniformly bounded 
functions such that
\(V \in \mathcal{V}\) and
\(\mathrm{Prob}\{\hat{V}_N \in \mathcal{V}\} \to 1\) as
\(N \to \infty\). 
Then 
\eqref{eq:inf-soc:q-continuity} holds in each
of the following cases:
\begin{enumerate}[label=\textup{(\roman*)},nosep]
\item \(\mathcal X\) and \(\Xi\) are
intervals, every \(W \in \mathcal V\) is nonincreasing, and, for every
\((x,u)\in\mathcal X\times\mathcal U\), the map
$F(x,u,\cdot)$
is either nondecreasing or nonincreasing;
\item 
\(\mathcal{X}\) is an interval, 
and $\mathcal{V}$ consists of functions
with common Lipschitz constant;
\item \(\mathcal{X}\) is convex
with nonempty interior, its dimension is at most three, and
$\mathcal{V}$ consists of convex functions with
common Lipschitz constant. 
\end{enumerate}
\end{lemma}

Condition~\textup{(i)} applies to one-dimensional models in which the
population and SAA value functions are monotone and the transition is
monotone in the noise. We use this condition in the
renewable harvesting application. Condition~\textup{(ii)}
applies when the state space is one-dimensional and the value functions
have a common Lipschitz bound. This is the condition used in the
inventory-control application. Condition~\textup{(iii)} covers
low-dimensional convex models whose value functions are uniformly
Lipschitz. The restriction to state dimension at most three reflects
the available entropy bounds for classes of convex functions.

In applications, the set \(\mathcal V\) is constructed from structural
properties of the population and empirical DP operators. Uniform
boundedness follows from the DP contraction property, while monotonicity,
Lipschitz continuity, or convexity is established from the model
dynamics and stage costs. These properties ensure that both \(V\) and
\(\hat V_N\) belong to the same deterministic class \(\mathcal V\),
with probability tending to one.

\subsection{Beyond uniqueness of optimal one-step costs and transitions}
\label{sec:nonunique-policy-limit}

For completeness, we record a supplementary extension of the 
functional CLT beyond \Cref{ass:uniqueness-type}. 
When population optimal
controls induce different one-step costs or transitions, the infimum
operator is only directionally differentiable, and the resulting limit
need not be Gaussian.  Proofs are deferred to
\Cref{app:nonunique-policy-limit}.

For \(q\in C(\mathcal X\times\mathcal U)\), we define the infimum operator
\begin{align}
\label{eq:inf-operator}
[\Psi(q)](x)
\coloneqq
\min_{u\in\mathcal U}q(x,u).
\end{align}
The theorem below requires that the population optimal solution
set
\(\mathcal U^\star \), defined in
\eqref{eq:solutionset}, be lower hemicontinuous, meaning that whenever
$x_N \in \mathcal{X}$, 
\(x_N\to x\), and \(u\in\mathcal U^\star(x)\), there exist
\(u_N\in\mathcal U^\star(x_N)\) such that \(u_N\to u\). 
As we show in the appendices, 
this lower hemicontinuity ensures  $\Psi$
is directionally differentiable at $\mathcal{Q}(V)$.
We define
\begin{align}
\label{eq:Phi}
\Phi(x,u,\xi)
\coloneqq
f(x,u,\xi)+\gamma V(F(x,u,\xi)),
\quad \text{and} \quad 
\Delta_N
\coloneqq
N^{1/2}
\left(
\hat{\mathcal Q}_N(\hat V_N)-\mathcal Q(V)
\right).
\end{align}

\begin{theorem}
\label{thm:nonunique-policy-limit}
Suppose that \Cref{inf-soc:ass-2}
holds.
Then
\begin{align}
\label{eq:Q-CLT}
N^{1/2}(\hat{\mathcal Q}_N(V)-\mathcal Q(V))
\rightsquigarrow
\mathfrak Y
\quad
\text{in} 
\quad 
C(\mathcal X\times\mathcal U),
\end{align}
where \(\mathfrak Y\) is a centered Gaussian 
process on $\mathcal{X} \times \mathcal{U}$ with
covariance kernel
\[
\operatorname{Cov}
\bigl(
\mathfrak Y(x,u),\mathfrak Y(x',u')
\bigr)
=
\operatorname{Cov}_{\xi\sim P}
\bigl(
\Phi(x,u,\xi),\Phi(x',u',\xi)
\bigr).
\]
If, furthermore,
\(\mathcal{U}^\star\) is lower hemicontinuous, 
\eqref{eq:inf-soc:q-continuity} holds, 
and
\begin{align}
\label{eq:nonunique-directional-linearization}
N^{1/2}
\big[
\Psi(\mathcal Q(V)+N^{-1/2}\Delta_N)
-
\Psi(\mathcal Q(V))
\big]
-
\Psi'(\mathcal Q(V);\Delta_N)
=
o_p(1)
\quad\text{in} \quad C(\mathcal X),
\end{align}
then
\[
N^{1/2}(\hat V_N-V)
\rightsquigarrow
\mathfrak Z
\quad
\text{in} \quad  C(\mathcal X),
\]
where \(\mathfrak Z\) is
the unique solution in \(C(\mathcal X)\) of
\[
\mathfrak Z(x)
=
\min_{u\in\mathcal{U}^\star(x)}
\left\{
\mathfrak Y(x,u)
+
\gamma
\mathbb{E}_{\xi\sim P}
[\mathfrak Z(F(x,u,\xi))]
\right\},
\qquad x\in\mathcal X.
\]
\end{theorem}

If \Cref{ass:uniqueness-type} also holds, the minimum in the limiting
fixed point equation is redundant, and the limit reduces to the
Gaussian limit in \eqref{eq:inf-soc:limit}.
The preceding theorem is stated under the linearization
condition \eqref{eq:nonunique-directional-linearization}.  It is the
nonunique control analogue of the linearization condition
\eqref{eq:inf-soc:linearization-error}. The next lemma provides
simple sufficient conditions.

\begin{lemma}[Compactness implies directional linearization]
\label{lem:nonunique-compact-directional-linearization}
If \Cref{inf-soc:ass-2} holds, 
\(\mathcal U^\star\) is lower
hemicontinuous, 
\(\mathcal A \) is
compact, and 
\eqref{eq:inf-soc:q-continuity} holds, 
then  condition
\eqref{eq:nonunique-directional-linearization} is satisfied.
\end{lemma}

\section{Empirical DP and policy optimization}
\label{sec:policy-saa}

To clarify the statistical effect of the fixed point structure, we compare
empirical DP with ordinary SAA applied directly to trajectory-based policy
optimization. In empirical DP, the one-period disturbance law is replaced by
its empirical distribution in the DP equation. In trajectory-based policy
optimization, complete disturbance trajectories are sampled and their average
discounted cost is minimized over a parameterized policy class.
We
first establish a CLT for SAA policy optimization and
then derive closed-loop identities comparing it with the 
fixed point limit of
empirical DP.
Proofs are deferred to
\Cref{app:variance-proofs}.

\subsection{Trajectory-based policy optimization}
\label{subsect:trajectory-policy-optimization}

We introduce an idealized SAA framework for policy optimization and study
its asymptotics. The formulation is idealized because it assumes
access to i.i.d.\ copies of the entire disturbance sequence
\(\boldsymbol\xi=(\xi_0,\xi_1,\ldots)\), so that each observation represents
one complete infinite-horizon sample path.

Fix \(x_0\in\mathcal X\). Let
\(\Theta\subseteq\mathbb R^q\) be nonempty and compact, and consider the
parameterized class
\(\Pi_\Theta\coloneqq\{\pi_\theta:\theta\in\Theta\}\) of stationary
Markov policies \(\pi_\theta:\mathcal X\to\mathcal U\).
For a disturbance trajectory
\(\boldsymbol\xi=(\xi_0,\xi_1,\ldots)\in\Xi^{\mathbb N_0}\), let
\(\boldsymbol x_0^\theta=x_0\) and define
\begin{equation}
\label{eq:parametric-closed-loop-process}
\boldsymbol x_{t+1}^\theta
=
F\bigl(
\boldsymbol x_t^\theta,
\pi_\theta(\boldsymbol x_t^\theta),
\xi_t
\bigr),
\qquad
t=0,1,2,\ldots.
\end{equation}
The realized total discounted cost under \(\pi_\theta\) is
\begin{equation}
\label{eq:parametric-trajectory-cost}
J(\theta;\boldsymbol\xi)
\coloneqq
\sum_{t\geq 0}
\gamma^t
f\bigl(
\boldsymbol x_t^\theta,
\pi_\theta(\boldsymbol x_t^\theta),
\xi_t
\bigr).
\end{equation}
Let \(P^\infty\coloneqq P^{\otimes\mathbb N_0}\) denote the distribution
of \(\boldsymbol\xi\). We consider the population policy optimization
problem
\begin{equation}
\label{eq:parametric-policy-population}
\min_{\theta\in\Theta}
\mathbb E_{\boldsymbol\xi\sim P^\infty}
\bigl[
J(\theta;\boldsymbol\xi)
\bigr].
\end{equation}
Given i.i.d.\ disturbance trajectories
\(\boldsymbol\xi^{(1)},\ldots,\boldsymbol\xi^{(N)}\) with common
distribution \(P^\infty\), the corresponding SAA optimal value is
\begin{equation}
\label{eq:parametric-policy-saa}
\hat\vartheta_N^{\mathrm{tr}}
\coloneqq
\min_{\theta\in\Theta}
\frac{1}{N}
\sum_{i=1}^{N}
J\bigl(
\theta;\boldsymbol\xi^{(i)}
\bigr).
\end{equation}
Thus, in contrast to empirical DP, the empirical distribution in
\eqref{eq:parametric-policy-saa} is placed directly on complete
disturbance trajectories.

\begin{assumption}[Identifiable, well-specified policy class, and Lipschitz trajectory cost]
\label{ass:parametric-policy-class}
\textup{(i)}
The mapping
\(J\) is jointly
measurable,
and there exists
$\tilde{\theta} \in \Theta$
such that 
\(\mathbb E_{\boldsymbol{\xi} \sim P^\infty}[J(\tilde{\theta}, \boldsymbol{\xi})]<\infty\).
\textup{(ii)}
There exists a measurable function
\(L:\Xi^{\mathbb N_0}\to[0,\infty)\) satisfying
\(\mathbb E_{\boldsymbol{\xi} \sim P^\infty}[L(\boldsymbol\xi)^2]<\infty\) such that,
for all $\boldsymbol{\xi} \in \Xi^{\mathbb{N}_0}$,
$J(\cdot; \boldsymbol{\xi})$
is Lipschitz continuous with Lipschitz constant
$L(\boldsymbol\xi)$.
\textup{(iii)}
The policy parameterization is identifiable: if
\(\pi_\theta=\pi_{\theta'}\), then \(\theta=\theta'\).
\textup{(iv)}
The policy class is well specified: there exists
\(\bar\theta\in\Theta\) such that \(\pi_{\bar\theta}=\pi^\star\).
Moreover, every solution of the population policy optimization problem
\eqref{eq:parametric-policy-population} represents \(\pi^\star\), that is,
\[
\operatorname*{argmin}_{\theta\in\Theta}
\mathbb E_{\boldsymbol\xi\sim P^\infty}
\bigl[
J(\theta;\boldsymbol\xi)
\bigr]
\subseteq
\{\theta\in\Theta:\pi_\theta=\pi^\star\}.
\]
\end{assumption}

The classical asymptotic theory for SAA optimal values shows that, when
the population problem has a unique solution, the limiting variance is
the variance of the random objective evaluated at that
population solution; see Theorem~5.7 in \cite{SDR}. 
Let us introduce this variance for our setting.
Let \((\boldsymbol x_t)_{t\geq0}\) denote the closed-loop process under
\(\pi^\star\), started from \(x_0\), that is,
\[
\boldsymbol x_{t+1}
=
F\bigl(
\boldsymbol x_t,
\pi^\star(\boldsymbol x_t),
\xi_t
\bigr),
\qquad
\boldsymbol x_0=x_0.
\]
We define
\begin{align}
\label{eq:direct-variance}
\sigma_{\pi^\star,\gamma}^2(x_0)
\coloneqq
\operatorname{Var}
\Big(
\sum_{t\geq 0}
\gamma^t
f\bigl(
\boldsymbol x_t,
\pi^\star(\boldsymbol x_t),
\xi_t
\bigr)
\Big).
\end{align}

\begin{proposition}[Policy optimization CLT]
\label{prop:policy-clt}
If
\Cref{inf-soc:ass-2,ass:parametric-policy-class} hold,
problem \eqref{eq:inf-soc} has a unique optimal policy \(\pi^\star\),
and $x_0 \in \mathcal{X}$, then
the optimal value of
\eqref{eq:parametric-policy-population} equals \(V(x_0)\), 
and
\begin{equation}
\label{eq:parametric-policy-clt}
N^{1/2}
\bigl(
\hat\vartheta_N^{\mathrm{tr}}-V(x_0)
\bigr)
\rightsquigarrow
\mathcal N
\bigl(
0,\sigma_{\pi^\star,\gamma}^2(x_0)
\bigr).
\end{equation}
\end{proposition}

\subsection{Closed-loop variance identities}
\label{subsect:variance-identities}

The preceding subsection identifies
\(\sigma_{\pi^\star,\gamma}^2(x_0)\) as the asymptotic variance of the
trajectory-based SAA optimal value. The fixed-point CLT, in turn, gives
the asymptotic variance
\(\sigma_{\mathrm{DP},\gamma}^2(x_0)\) 
from \eqref{eq:DP-variance} of the SAA DP optimal value.
We now express both variances in terms of a covariance kernel.

Recall that \((\boldsymbol x_t)_{t\geq0}\) denotes the closed-loop process
under \(\pi^\star\), started from \(x_0\), and let
\((\tilde{\boldsymbol x}_t)_{t\geq0}\) be an independent copy. We take both
processes to be independent of the Gaussian limit \(\mathfrak H\) in
\eqref{eq:inf-soc:clt-mathcalV}, and let \(\mu_t\) denote the law of
\(\boldsymbol x_t\). Let \(\mathcal K_\gamma\) be the covariance kernel
defined in \eqref{covHt}, namely,
\[
\mathcal K_\gamma(x,x')
\coloneqq
\operatorname{Cov}_{\xi\sim P}
\bigl(
\Phi^\star(x,\xi),
\Phi^\star(x',\xi)
\bigr).
\]
The subscript \(\gamma\) emphasizes that
\(\mathcal K_\gamma\) depends on the discount factor through
\(\Phi^\star\).
For \(\rho\in(0,1)\), we define the discounted state-occupation measure
\[
\nu_\rho
\coloneqq
(1-\rho)\sum_{t\geq 0}\rho^t\mu_t .
\]
This is the normalized discounted state frequency of the closed-loop
process; cf.\ Remark~6.3.1 in \cite{HernandezLerma1996}. In finite
state-action models, the variance of the total discounted cost identity below 
follows from the 
classical covariance formula in Theorem~1 of \cite{Sobel1982}.

\begin{proposition}[Closed-loop variances]
\label{prop:random-path-variance-ratio}
If \Cref{inf-soc:ass-2,ass:uniqueness-type} hold,
\eqref{eq:inf-soc:fixed-point-linearization} is satisfied, 
problem \eqref{eq:inf-soc}
has a unique optimal policy $\pi^\star$, and
  $x_0 \in \mathcal{X}$, 
then
\begin{align}
\label{eq:DP-variance-discounted-occupation}
\sigma_{\mathrm{DP},\gamma}^2(x_0)
& =
\sum_{s, t \geq 0}
\gamma^{s+t}
\mathbb{E}
\left[
\mathcal{K}_\gamma(\boldsymbol x_s,\tilde{\boldsymbol x}_t)
\right]
= 
\frac{1}{(1-\gamma)^2}
\int_{\mathcal X\times\mathcal X}
\mathcal K_\gamma(x,x')
\,\nu_\gamma(\mathrm dx)\nu_\gamma(\mathrm dx'),
\\ 
\label{eq:direct-variance-discounted-occupation}
\sigma_{\pi^\star,\gamma}^2(x_0)
& =
\sum_{t \geq 0}
\gamma^{2t}
\mathbb{E}
\left[
\mathcal{K}_\gamma(\boldsymbol x_t,\boldsymbol x_t)
\right]
= 
\frac{1}{1-\gamma^2}
\int_{\mathcal X}
\mathcal K_\gamma(x,x)\,\nu_{\gamma^2}(\mathrm dx).
\end{align}
\end{proposition}

\Cref{prop:random-path-variance-ratio} shows that the two variances
average the same covariance kernel in different ways. The empirical-DP
optimal-value variance averages \(\mathcal K_\gamma(x,x')\) over two
independent states drawn from the discounted occupation measure
\(\nu_\gamma\). Equivalently, it contains the cross-time terms
\(\mathbb{E}[\mathcal K_\gamma(\boldsymbol x_s,
\tilde{\boldsymbol x}_t)]\), and therefore includes both diagonal and
off-diagonal evaluations of the kernel. The trajectory-based SAA
optimal-value variance uses only the diagonal
\(x\mapsto\mathcal K_\gamma(x,x)\), averaged with respect to
\(\nu_{\gamma^2}\). Equivalently, it contains only the terms
\(\mathbb{E}[\mathcal K_\gamma(\boldsymbol x_t,\boldsymbol x_t)]\).

Although we do not pursue this direction here, the variance expressions in \Cref{prop:random-path-variance-ratio} may suggest that, under suitable conditions, the normalized variances \[ (1-\gamma)^2\sigma_{\mathrm{DP},\gamma}^2(x_0) \quad\text{and}\quad (1-\gamma^2)\sigma_{\pi^\star,\gamma}^2(x_0) \] converge as \(\gamma \uparrow 1\). Such limits may be connected to vanishing-discount theory for SOC problems, which studies conditions under which discounted SOC problems converge, in an appropriate sense, to average-cost SOC problems; see \cite{Feinberg2012,HernandezLerma1996}.

\subsection{A finite state-action example with three variance-ratio regimes}
\label{subsect:variance-ratio-example}

The following finite state-action example shows that the variance ratio associated
with \eqref{eq:direct-variance} and \eqref{eq:DP-variance} can be less
than, equal to, or greater than one, depending on a 
model parameter and the discount factor.

\begin{example}
\label{ex:param-finite-state-variance-ratio}
\normalfont
Let \(\mathcal X=\{0,1,2\}\),
\(\mathcal U=\{0\}\), and \(x_0=0\). The deterministic dynamics are
\(0\mapsto1\), \(1\mapsto2\), and \(2\mapsto1\). Let
\(\xi=(Y,Z)\), where \(Y\) and \(Z\) are independent Rademacher random
variables, and let \(\eta\in\mathbb R\). 
We define
\(f_\eta(0,0,\xi)\coloneqq Y\), \(f_\eta(1,0,\xi)\coloneqq \eta Z\), and
\(f_\eta(2,0,\xi) \coloneqq \eta^2Z\). 
\end{example}

In the next result, we make the dependence of the limiting variances on the parameter \(\eta\) explicit.

\begin{lemma}[Variance formulas for
\Cref{ex:param-finite-state-variance-ratio}]
\label{lem:param-finite-state-variances}
For every \(\eta\in\mathbb R\) and \(\gamma\in(0,1)\),
\begin{equation*}
\begin{aligned}
\sigma_{\mathrm{DP},\eta,\gamma}^2(0)
=
1+
\frac{\gamma^2\eta^2(1+\gamma \eta)^2}
     {(1-\gamma^2)^2},
\quad \text{and} \quad 
\sigma_{\pi^\star,\eta,\gamma}^2(0)
=
1+
\frac{\gamma^2\eta^2+\gamma^4\eta^4}
     {1-\gamma^4}.
\end{aligned}
\end{equation*}
\end{lemma}

For 
\(\eta \in \{-1, 0, 2\}\),
\Cref{fig:param-finite-state-variance-ratio}
illustrates the variance ratio
\begin{align}
\label{eq:variance-ratio}
R_\eta(\gamma)
\coloneqq
\frac{\sigma_{\pi^\star,\eta,\gamma}^2(0)}
{\sigma_{\mathrm{DP},\eta,\gamma}^2(0)}.
\end{align}
For \(\eta = 0\), the ratio remains constant; for \(\eta = 2\), it decreases to zero; and for \(\eta = -1\), it diverges to \(\infty\) as \(\gamma \uparrow 1\).

\begin{figure}[t]
\centering
\begin{subfigure}[t]{0.48\textwidth}
\centering
\begin{tikzpicture}
\begin{axis}[
    width=0.95\linewidth,
    height=0.72\linewidth,
    xlabel={Discount factor \(\gamma\)},
    ylabel={Variance ratio \(R_\eta(\gamma)\)},
    xmin=0,
    xmax=1,
    ymin=0,
    ymax=1.05,
    domain=0:1,
    samples=300,
    grid=major,
    tick label style={font=\small},
    label style={font=\small},
    legend style={
        font=\small,
        at={(0.03,0.03)},
        anchor=south west,
        draw=none,
        fill=white,
        fill opacity=0.85,
        text opacity=1
    },
]
\addplot+[
    mark=none,
    thick,
    color={tab-blue},
]
{1};
\addlegendentry{\(\eta=0\)}

\addplot+[
    mark=none,
    very thick,
    dashed,
    color={tab-orange},
]
{
    ((1 - x^2)*(1 + 4*x^2 + 15*x^4))
    /
    ((1 + x^2)*(1 + 2*x^2 + 16*x^3 + 17*x^4))
};
\addlegendentry{\(\eta=2\)}

\addplot+[
    only marks,
    mark=*,
    mark size=1.6pt,
    color={tab-orange},
]
coordinates {(1,0)};

\end{axis}
\end{tikzpicture}
\caption{Cases \(\eta=0\) and \(\eta=2\).}
\label{fig:param-ratio-constant-decreasing}
\end{subfigure}
\hfill
\begin{subfigure}[t]{0.48\textwidth}
\centering
\begin{tikzpicture}
\begin{axis}[
    width=0.95\linewidth,
    height=0.72\linewidth,
    xlabel={Discount factor \(\gamma\)},
    ylabel={Variance ratio \(R_{\eta}(\gamma)\)},
    xmin=0,
    xmax=1,
    ymin=0.8,
    ymax=1000,
    ymode=log,
    domain=0:0.9995,
    samples=300,
    grid=major,
    clip=false,
    tick label style={font=\small},
    label style={font=\small},
    legend style={
        font=\small,
        at={(0.03,0.97)},
        anchor=north west,
        draw=none,
        fill=white,
        fill opacity=0.85,
        text opacity=1
    },
]
\addplot+[
    mark=none,
    very thick,
    densely dotted,
    color={tab-brown},
]
{
    (
        1 + (x^2 + x^4)/(1 - x^4)
    )
    /
    (
        1 + x^2*(1 - x)^2/(1 - x^2)^2
    )
};
\addlegendentry{\(\eta=-1\)}

\end{axis}
\end{tikzpicture}
\caption{Case \(\eta=-1\).}
\label{fig:param-ratio-increasing}
\end{subfigure}
\caption{Three variance ratio regimes for
\Cref{ex:param-finite-state-variance-ratio}. 
The function $R_\eta(\gamma)$ is the ratio of the 
variance of total discounted cost to the
asymptotic variance of the SAA DP optimal value, 
as defined in \eqref{eq:variance-ratio}.}
\label{fig:param-finite-state-variance-ratio}
\end{figure}
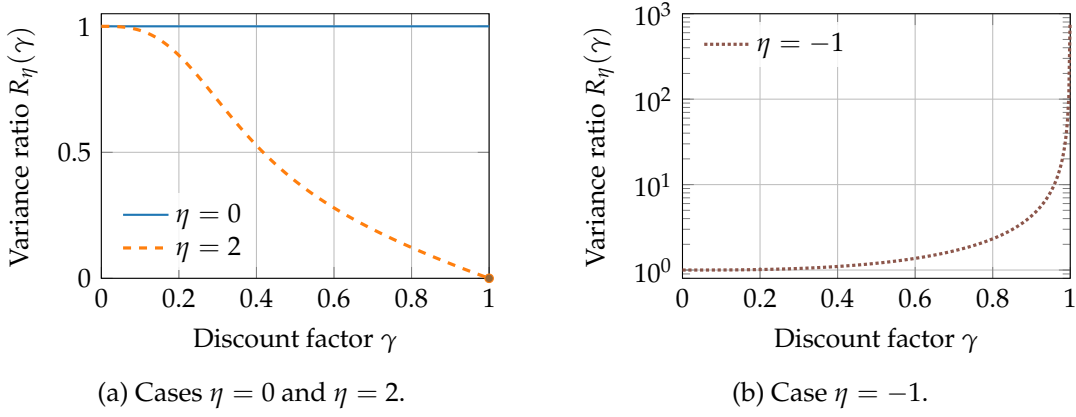

\section{Application to inventory control}
\label{sec:inventory-control}

We consider a risk-neutral discounted inventory model. In our
formulation, \(a_t\ge0\) is the physical order quantity. The state
\(x_t\in\mathbb R\) is inventory on hand minus backlog, and negative
values represent backlog. The demands \((D_t)_{t\ge0}\) are i.i.d.\ with common distribution function \(H\).
Moreover, the demand \(D\) has a density  supported on
\([0,\overline D]\) and  \(H\) is
strictly increasing on \([0,\overline D]\).
Let \([r]_+\coloneqq\max\{r,0\}\).
The inventory control problem 
with initial inventory $x_0 \in \mathbb{R}$
is given by
(see, for example, \cite{zipkin})
\[
\inf_{a_t\ge0}
\mathbb{E}\bigg[
\sum_{t\geq 0}\gamma^t
\left(
ca_t+b[D_t-(x_t+a_t)]_+
+h[x_t+a_t-D_t]_+
\right)
\bigg]
\quad \text{s.t.} \quad 
x_{t+1}=x_t+a_t-D_t .
\]
The constants \(c,b,h\ge0\) are ordering, backlog, and holding costs,
and we assume \(b>c\) and $c + h >0$.

Under our assumptions on the demand $D$, 
the optimal policy is the
base-stock policy \(a^\star(x)=[x^\star-x]_+\), where
\[
x^\star
=
H^{-1}
\left(
\frac{b-(1-\gamma)c}{b+h}
\right).
\]
Since \(H\) is strictly increasing, \(x^\star\) is unique, and
\(x^\star\in[0,\overline D]\).

\subsection{Post-order  reformulation}

The original order quantity \(a_t\) is unbounded and the state space is
unbounded. 
We therefore introduce a reformulation. Let
$\underline x\coloneqq\min\{x_0,-\overline D\}$, 
$\overline x\coloneqq\max\{x_0,\overline D\}$, 
and set
\[
\mathcal X=[\underline x,\overline x],
\quad
\mathcal U=[0,\overline x],
\quad \text{and} \quad 
\Xi=[0,\overline D].
\]
Given \(x\in\mathcal X\) and \(u\in\mathcal U\), the target \(u\)
induces the physical order \([u-x]_+\) and the post-order inventory
\(\max\{x,u\}\). We define
\begin{align*}
F(x,u,d) & \coloneqq \max\{x,u\}-d, 
\\ 
f(x,u,d)
& \coloneqq 
c[u-x]_+
+b[d-\max\{x,u\}]_+
+h[\max\{x,u\}-d]_+ .
\end{align*}
The post-inventory problem is
\[
\inf_{u_t\in \mathcal{U}}
\mathbb{E}\left[
\sum_{t\geq 0}\gamma^t f(x_t,u_t,D_t)
\right]
\quad \text{s.t.} \quad 
x_{t+1}=F(x_t,u_t,D_t).
\]

This target-level formulation has the same optimal value as the
original problem. Every target \(u\in[0,\overline x]\) induces the
feasible physical order \([u-x]_+\ge0\). Conversely, the optimal
base-stock policy in the original problem is represented by the
constant target \(u^\star(x)\equiv x^\star\), since
\(x^\star\in[0,\overline D]\subseteq[0,\overline x]\), and it induces the
physical order \([x^\star-x]_+\). Hence restricting to bounded targets does
not change the optimal value.

We now verify \Cref{inf-soc:ass-2}. The sets
\(\mathcal X\), \(\mathcal U\), and \(\Xi\) are compact. For
\(x\in\mathcal X\), \(u\in\mathcal U\), and \(d\in\Xi\), we have
\(F(x,u,d)\le\overline x\). Also, \(\max\{x,u\}\ge0\), because
\(u\ge0\), and hence
\(F(x,u,d)\ge-\overline D\ge\underline x\). Therefore
\(F(\mathcal X\times\mathcal U\times\Xi)\subseteq\mathcal X\).
The functions \(F\) and \(f\) are 
Lipschitz with deterministic Lipschitz constants. Thus
\Cref{inf-soc:ass-2} holds.

From \eqref{eq:solutionset},
we recall that $\mathcal{U}^\star$
denotes the optimal solution set
associated with the DP principle.
Next, we verify \Cref{ass:uniqueness-type}.

\begin{lemma}
\label{lem:inventory-redundant-targets}
For every  \(x\in\mathcal X\) and every \(u\in\mathcal{U}^\star(x)\),
$
\max\{x,u\}=\max\{x,x^\star\}
$.
Consequently, \Cref{ass:uniqueness-type}
holds with 
$f^\star(x,D) = f(x, x^\star, D)$,
and 
$F^\star(x,D) = F(x, x^\star, D)$.
\end{lemma}

\begin{proof}
The control \(u\) affects the objective only through the
post-order inventory \(\max\{x,u\}\). By the base-stock optimality of
the original inventory problem, the unique optimal post-order inventory
is \(\max\{x,x^\star\}\). Thus every optimal control must induce this
post-order inventory. 
\end{proof}

\subsection{Inventory CLT}

Now, we formulate the CLT for the post-order formulation. Since we have already verified \Cref{inf-soc:ass-2,ass:uniqueness-type}, 
\Cref{thm:inf-soc:clt} provides
the limit law  $\mathfrak{H}$ in
\eqref{eq:inf-soc:clt-mathcalV}.

\begin{proposition}[Inventory  CLT]
\label{prop:inventory-clt}
For the post-order inventory reformulation,
\[
N^{1/2}(\hat V_N-V)
\rightsquigarrow
\mathfrak G
\quad
\text{in}
\quad C(\mathcal X),
\]
where $\mathfrak{G}$ is the unique centered Gaussian 
process on $\mathcal{X}$ satisfying
\[
\mathfrak G(x)
=
\mathfrak H(x)
+
\gamma\mathbb{E}_{D \sim P}[\mathfrak G(\max\{x,x^\star\}-D)],
\qquad
x\in\mathcal X.
\]
\end{proposition}

\begin{proof}
We verify the hypotheses of the CLT in 
\Cref{thm:inf-soc:clt}.
We have already verified 
\Cref{inf-soc:ass-2,ass:uniqueness-type}.
For the latter, see \Cref{lem:inventory-redundant-targets}.

Now, we verify the linearization
condition in \eqref{eq:inf-soc:linearization-error}
using \Cref{prop:main-linearization-conditions,lemma:additive-noise-transition-linearization}.
To this end,
we verify the compactness condition of the operator $\mathcal{A}$ defined in \eqref{eq:A-operator}. 
For our reformulation, the operator $\mathcal{A}$ is
 given by
\[
[\mathcal A W](x,u)
=
\mathbb{E}_{D\sim P}[W(\max\{x,u\}-D)].
\]
Since the demand $D$ has a density, \Cref{lemma:additive-noise-transition-linearization} 
ensures compactness of $\mathcal{A}$.

Next, we verify the stochastic equicontinuity-type
condition \eqref{eq:inf-soc:continuity}
using \Cref{prop:main-equicontinuity-conditions,lem:Lipschitz-equicontinuity}. 
Since \(f\) is bounded on the compact
set \(\mathcal X\times\mathcal U\times\Xi\), let
\(M_f\coloneqq\|f\|_\infty\). The contraction bound 
ensures
$V$ and, with probability one,
$\hat{V}_N$ have a sup norm of
at most $(1-\gamma)^{-1} M_f$.
Moreover, standard fixed point
arguments ensure that
$V$ and, with probability one, $\hat{V}_N$
are Lipschitz continuous
with Lipschitz constant
$(1-\gamma)^{-1}(c+b+h)$.
Thus \(V,\hat V_N\) belong,
with probability one, to the set
\[
\mathcal V
=
\Big\{
W\in C(\mathcal X):
\|W\|_\infty\le(1-\gamma)^{-1} M_f,
\quad
\operatorname{Lip}(W)\le
(1-\gamma)^{-1}(c+b+h)
\Big\},
\]
where $\operatorname{Lip}(W)$ denotes
the Lipschitz constant of $W \in C(\mathcal{X})$.
Since \(\mathcal X\) is one-dimensional,
\Cref{lem:Lipschitz-equicontinuity} implies
\eqref{eq:inf-soc:q-continuity}.
Moreover, $\mathcal{A}$ is compact,
so \Cref{prop:main-equicontinuity-conditions}
applies.

Having verified the hypotheses of 
\Cref{thm:inf-soc:clt}, we obtain the
CLT.
\end{proof}

\subsection{Trajectory-based policy optimization over base-stock policies}
\label{subsect:inventory-policy-SAA}

For the original inventory formulation, 
we next consider trajectory-based SAA over the one-dimensional class
of base-stock policies.  For \(s\in[0,\overline D]\),
we define
\(a_s(x)\coloneqq [s-x]_+\).  For a demand trajectory
\(\boldsymbol D=(D_0,D_1,\ldots)\), let
\(J(s;\boldsymbol D)\) denote the realized total discounted cost of the
original inventory model under \(a_s\), starting from \(x_0\).

\begin{proposition}
\label{prop:inventory-parametric-policy-class}
If $x_0 \leq x^\star$, then
the  policy class
\(\{a_s:s\in[0,\overline D]\}\) satisfies
\Cref{ass:parametric-policy-class}
with $\Theta = [0,\overline D]$.
\end{proposition}

\begin{proof}
For \Cref{ass:parametric-policy-class}\textup{(i)}, let
\(\boldsymbol x_t^s\) be the inventory state under \(a_s\). Since
\(x_0\le x^\star\le\overline D\), \(s\in[0,\overline D]\), and
\(D_t\in[0,\overline D]\), we have
\(\boldsymbol x_t^s\in[-\overline D,\overline D]\) for \(t\ge1\).
The order quantities and one-period costs are therefore bounded by a
deterministic constant, uniformly in \(s\) and the demand trajectory.
Hence \(J(s;\boldsymbol D)\) is integrable for, in particular,
\(s=x^\star\). Moreover, every finite-horizon discounted cost is jointly
measurable in \((s,\boldsymbol D)\), and \(J\) is its pointwise limit.
Thus \Cref{ass:parametric-policy-class}\textup{(i)} holds.

For \Cref{ass:parametric-policy-class}\textup{(ii)}, couple the systems
under \(a_s\) and \(a_r\) using the same demand trajectory. The recursion
\(\boldsymbol x_{t+1}^s=\max\{\boldsymbol x_t^s,s\}-D_t\) and the
nonexpansiveness of the maximum imply inductively that
\(\lvert\boldsymbol x_t^s-\boldsymbol x_t^r\rvert\le\lvert s-r\rvert\).
Consequently,
\(\lvert a_s(\boldsymbol x_t^s)-a_r(\boldsymbol x_t^r)\rvert
\le2\lvert s-r\rvert\).
Since \(y\mapsto b[d-y]_++h[y-d]_+\) is Lipschitz with constant at most
\(b+h\), the difference between the one-period costs is at most
\((2c+b+h)\lvert s-r\rvert\). Therefore,
\[
\lvert J(s;\boldsymbol D)-J(r;\boldsymbol D)\rvert
\le
\frac{2c+b+h}{1-\gamma}\lvert s-r\rvert .
\]
Thus \Cref{ass:parametric-policy-class}\textup{(ii)} holds with the
deterministic, and hence square-integrable, Lipschitz modulus
\(L(\boldsymbol D)=(2c+b+h)/(1-\gamma)\).

For \Cref{ass:parametric-policy-class}\textup{(iii)}, if \(a_s=a_r\),
then \(s=a_s(0)=a_r(0)=r\). Hence the parameterization is identifiable.

For \Cref{ass:parametric-policy-class}\textup{(iv)}, the inventory
structural result gives \(a_{x^\star}=\pi^\star\), so the policy class is
well specified. It remains to identify the population minimizers. 
Let
\(\Upsilon(s)\coloneqq \mathbb E[J(s;\boldsymbol D)]\). For \(s\ge x_0\), 
we have
\[
\Upsilon(s)
=
c(s-x_0)
+\frac{1}{1-\gamma}
\mathbb{E}\!\left[
b(D-s)_+ + h(s-D)_+
\right]
+\frac{\gamma c}{1-\gamma}\mathbb{E}[D]
\]
and hence
\[
\Upsilon'(s)
=
\frac{b+h}{1-\gamma}\bigl(H(s)-H(x^\star)\bigr).
\]
Because \(H\) is strictly increasing, \(\Upsilon\) is strictly decreasing on
\([x_0,x^\star]\) and strictly increasing on
\([x^\star,\overline D]\).

For \(s<x_0\), if \(x_0<x^\star\), then \(a_s\) does not order the
initial inventory to the uniquely optimal level \(x^\star\). If
\(x_0=x^\star\), then on the event \(\{D_0>0\}\), which has positive
probability, the next state is below \(x^\star\), while \(a_s\) selects
a post-order inventory strictly below \(x^\star\). Thus \(a_s\) is
strictly suboptimal in either case. Hence
\(\operatorname*{argmin}_{s\in[0,\overline D]}\Upsilon(s)=\{x^\star\}\),
and every population minimizer represents \(\pi^\star\). This verifies
\Cref{ass:parametric-policy-class}\textup{(iv)}.
\end{proof}

The trajectory-SAA optimal value is
\[
\hat\vartheta_N^{\mathrm{tr}}
\coloneqq
\min_{s\in[0,\overline D]}
\frac1N
\sum_{i=1}^N
J(s;\boldsymbol D^{(i)}),
\]
where
\(\boldsymbol D^{(1)},\ldots,\boldsymbol D^{(N)}\) are independent
demand trajectories.
\Cref{prop:inventory-parametric-policy-class,prop:policy-clt}
imply the following CLT.

\begin{corollary}[Trajectory-SAA CLT over base-stock policies]
\label{cor:inventory-policy-SAA}
If $x_0 \leq x^\star$, then
\[
N^{1/2}
\bigl(
\hat\vartheta_N^{\mathrm{tr}}-V(x_0)
\bigr)
\rightsquigarrow
\mathcal N
\bigl(
0,\sigma_{\pi^\star,\gamma}^2(x_0)
\bigr).
\]
\end{corollary}

\subsection{Variance comparisons}

For \(x_0\le x^\star\), combining eq.~(3.9) in \cite{cltshapcheng21} with
\Cref{prop:random-path-variance-ratio}, we can show that
\[
\sigma_{\pi^\star,\gamma}^2(x_0)
=
\frac{1-\gamma}{1+\gamma}
\cdot 
\sigma_{\mathrm{DP},\gamma}^2(x_0).
\]
For \(\gamma\in\{0.3,0.6,0.9\}\),
\Cref{fig:inventory-clt} visualizes the limiting Gaussian
distributions for an example with \(x_0=0\), \(c=1\), \(b=4\),
\(h=1\), and \(D=\overline D\,\mathrm{Beta}(2,3)\), where
\(\overline D=6\).

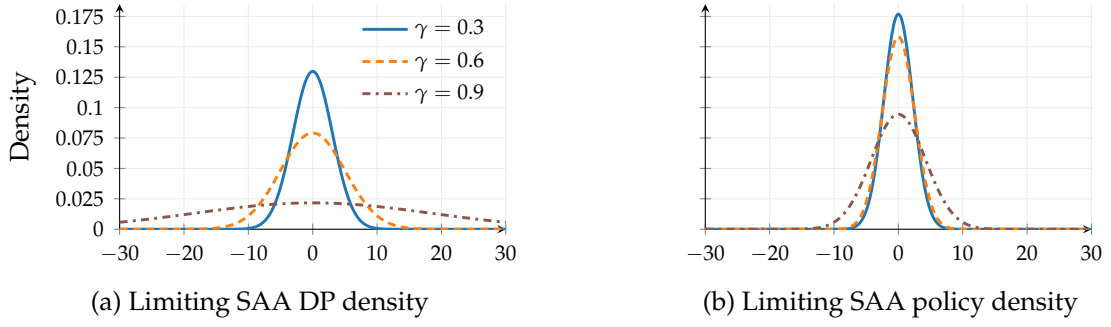
\begin{figure}[t]
\centering

\begingroup

\pgfplotsset{
  inventory normal axis/.style={
    width=0.455\textwidth,
    height=0.31\textwidth,
    xmin=-30,
    xmax=30,
    ymin=0,
    ymax=0.185,
    domain=-30:30,
    samples=301,
    axis lines=left,
    enlargelimits=false,
    clip=true,
    grid=major,
    major grid style={
      draw=gray!15
    },
    xtick={
      -30,-20,-10,0,10,20,30
    },
    ytick={
      0,
      0.025,
      0.050,
      0.075,
      0.100,
      0.125,
      0.150,
      0.175
    },
    scaled y ticks=false,
    y tick label style={
      /pgf/number format/fixed,
      /pgf/number format/precision=3
    },
    tick label style={
      font=\scriptsize
    },
    label style={
      font=\small
    },
    legend style={
      font=\scriptsize,
      draw=none,
      fill=none,
      at={(0.98,0.98)},
      anchor=north east,
      row sep=1pt
    },
    legend cell align=left
  },
  gamma03/.style={
    no marks,
    line width=1.1pt,
    solid,
    color=tab-blue
  },
  gamma06/.style={
    no marks,
    line width=1.1pt,
    densely dashed,
    color=tab-orange
  },
  gamma09/.style={
    no marks,
    line width=1.1pt,
    dash dot,
    color=tab-brown
  }
}

\subfloat[
  Limiting SAA DP density
  \label{fig:inventory-dynamic-normal}
]{%
  \begin{tikzpicture}
    \begin{axis}[
      inventory normal axis,
      ylabel={Density}
    ]

      \addplot[gamma03]
      {
        exp(-x^2/(2*9.4449))/
        sqrt(2*pi*9.4449)
      };
      \addlegendentry{\(\gamma=0.3\)}

      \addplot[gamma06]
      {
        exp(-x^2/(2*25.4042))/
        sqrt(2*pi*25.4042)
      };
      \addlegendentry{\(\gamma=0.6\)}

      \addplot[gamma09]
      {
        exp(-x^2/(2*338.9693))/
        sqrt(2*pi*338.9693)
      };
      \addlegendentry{\(\gamma=0.9\)}

    \end{axis}
  \end{tikzpicture}%
}
\hfill
\subfloat[
  Limiting SAA policy density
  \label{fig:inventory-trajectory-normal}
]{%
  \begin{tikzpicture}
    \begin{axis}[
      inventory normal axis,
      yticklabels=\empty
    ]

      \addplot[gamma03]
      {
        exp(-x^2/(2*5.0857))/
        sqrt(2*pi*5.0857)
      };

      \addplot[gamma06]
      {
        exp(-x^2/(2*6.3511))/
        sqrt(2*pi*6.3511)
      };

      \addplot[gamma09]
      {
        exp(-x^2/(2*17.8405))/
        sqrt(2*pi*17.8405)
      };

    \end{axis}
  \end{tikzpicture}%
}

\caption{
Mean-zero Gaussian densities for the inventory model.
Panel~\textup{(a)} shows the densities with variances
\(\sigma_{\mathrm{DP},\gamma}^2(x_0)\), while
Panel~\textup{(b)} shows the densities with variances
\(\sigma_{\pi^\star,\gamma}^2(x_0)\).
}
\label{fig:inventory-clt}

\endgroup

\end{figure}

\section{Application to renewable harvesting}
\label{sec:renewable-resource-harvesting}
 
Renewable harvesting is a classical application of SOC \cite{EkerhovdFlamSteinshamn,GaigiLyVathScotti,NyassokeSadefoFono}. In this section, we consider the stochastic stock-recruitment model of \cite{Reed}. This section demonstrates how the general framework applies to a renewable harvesting model with multiplicative  uncertainty, nonlinear harvesting costs, and a concave stock-recruitment function.
 
In the renewable harvesting model of \cite{Reed}, the biomass
evolves according to $x_{t+1}=\xi_tg(x_t-h_t)$
where $x_t\in[0,\infty)$ denotes the biomass at time $t$,
$h_t \in [0, x_t]$ is the harvest amount, 
and
 $(\xi_t)$ is an
i.i.d.\ sequence supported on $[a,b]$
with $0<a<1<b$, unit mean, and a density on $[a,b]$. 
Moreover,  
the stock-recruitment function
\(g:[0,\infty)\to[0,\infty)\) is  continuous,
strictly concave, nondecreasing, and continuously differentiable on
\((0,\infty)\).  We denote by \(g'(0)\) its
right-hand derivative at \(0\), which exists in
\([0,\infty]\) by concavity and monotonicity. Following \cite{Reed}, we
also assume
\[
bg'(0)>1,
\quad 
\limsup_{x \to \infty} b g(x)/x < 1, 
\quad \text{and} \quad 
\lim_{x\to\infty}ag'(x)<1 .
\]
These conditions imply the
existence of finite numbers $r\coloneqq \max\{x:ag(x)\ge x\}$ and
$\overline{x}\coloneqq \max\{x:bg(x)\ge x\}$ such that, under zero harvesting, \(h_t=0\) for all \( t \geq 0\), 
the process \((x_t)_{t\ge0}\)  is recurrent in
$[r,\overline{x}]$;
see p.~351 in \cite{Reed}.

The harvesting model uses the unit harvesting cost $c(x) \coloneqq k/x^\theta$, where $k>0$ is a
cost parameter,
and $\theta\in(0,1)$. The zero-profit biomass level
$x_\infty=(k/p)^{1/\theta}$ is characterized by $c(x_\infty)=p$, 
where  $p>0$ is the unit selling price.
We assume $x_\infty\le r$. Since $g(0)\ge0$, the concavity of
$x\mapsto ag(x)-x$ implies its nonnegativity on the interval $[0,r]$,
so this assumption is equivalent to $ag(x_\infty)\ge x_\infty$.

We define the immediate harvesting reward
$R(x)\coloneqq p(x-x_\infty)-\int_{x_\infty}^x c(t)\,\mathrm{d}t$.
For $x_0 \in [x_\infty, \overline{x}]$ and $\gamma\in(0,1)$, the harvesting
problem is
\[
\max_{0 \le h_t \le x_t}
\mathbb{E}\!\left[\sum_{t\geq 0}\gamma^t
\big(R(x_t) - R(x_t - h_t)\big)\right]
\quad \text{s.t.} \quad  
x_{t+1} = \xi_t\,g(x_t - h_t).
\]

The following condition is a strict version of the monotonicity condition
used in \cite{Reed} to obtain a base-stock-type optimal policy;
see eq.~(3.4) and Theorem~1 in \cite{Reed}. Throughout the section, 
we assume that,  for every $\xi\in[a,b]$, 
\begin{equation}
\label{eq:reed-condition-A}
x \mapsto
\frac{p-c(\xi g(x))}{p-c(x)}g'(x)
\quad
\text{is nonnegative and strictly decreasing on }(x_\infty,\overline{x}].
\end{equation}
This condition says that the marginal effect of
retaining additional post-decision stock is positive but diminishing,
uniformly over the recruitment shock. 
We can show that \eqref{eq:reed-condition-A} is satisfied for  \(g(x)=\bar ax/(1+\bar bx)\) with constants $\bar a$, $\bar b > 0$, provided the
standing assumptions above hold;
cf.\ p.~356 in \cite{Reed}.

\subsection{Target-escapement reformulation}

By Theorem~1 in \cite{Reed}, under condition
\eqref{eq:reed-condition-A}, the original problem admits a base-stock-type
optimal policy \(h^\star(x)=[x-x^\star]_+\), where
\(x^\star\in[x_\infty,\overline x]\). 
We use this structural
result to reformulate the harvesting problem in a form amenable to our
theory.

The original harvest quantity \(h_t\) is constrained by the current
biomass \(x_t\). We therefore introduce a target-escapement
formulation. Let
\[
\mathcal X=[x_\infty,\overline x],
\quad
\mathcal U=[x_\infty,\overline x],
\quad \text{and} \quad 
\Xi=[a,b].
\]
Given \(x\in\mathcal X\) and \(u\in\mathcal U\), the target \(u\)
induces the post-harvest stock \(\min\{x,u\}\) and the harvest
\([x-u]_+\). Since the general framework of \Cref{sec-basic} is a
minimization framework, we use the negative one-period reward as the
stage cost. We define
\begin{align}
\label{eq:harvest-fF}
F(x,u,\xi)\coloneqq \xi g(\min\{x,u\})
\quad \text{and} \quad 
f(x,u,\xi)\coloneqq R(\min\{x,u\})-R(x).
\end{align}
The target-escapement problem is
\[
\inf_{u_t\in\mathcal U}
\mathbb{E}
\Big[
\sum_{t\geq 0}
\gamma^t f(x_t,u_t,\xi_t)
\Big]
\quad \text{s.t.} \quad
x_{t+1}=F(x_t,u_t,\xi_t).
\]

This target-escapement formulation restricts post-harvest stocks to
\([x_\infty,\overline x]\), whereas the original harvesting problem allows
post-harvest stocks in \([0,x]\). This restriction does not change the
optimal value for initial states \(x\in\mathcal X\). Indeed, by
Theorem~1 in \cite{Reed}, the original problem admits an optimal
constant-escapement policy \(h^\star(x)=[x-x^\star]_+\) with
\(x^\star\in[x_\infty,\overline x]\). The induced post-harvest stock is
\(s^\star(x)=\min\{x,x^\star\}\ge x_\infty\). Thus an optimal policy of
the original problem is feasible in the restricted target formulation.
Conversely, every target \(u\in\mathcal U=[x_\infty,\overline x]\)
induces the feasible harvest \([x-u]_+\) for the original problem. Hence
the restricted target-escapement formulation has the same optimal value
as the original harvesting problem, up to the change of sign.

Before verifying \Cref{inf-soc:ass-2}, let us note that
$bg(\overline x)=\overline x$. Indeed, 
if \(bg(\overline x)>\overline x\), then the continuity of \(x\mapsto bg(x)-x\)
implies \(bg(x)>x\) for some \(x>\overline x\), contradicting the maximality of \(\overline x\).

\begin{lemma}
\label{lem:harvest-assumption-1}
For the target-escapement formulation,
\Cref{inf-soc:ass-2} holds.
\end{lemma}
\begin{proof}
(i)~The sets $\mathcal{X}$, $\mathcal{U}$, and $\Xi$ are compact  intervals. 

(ii)~The functions $f$ and $F$ are continuous because $R$ and $g$ are continuously differentiable and $(x,u)\mapsto\min\{x,u\}$ is continuous.
Let $x$, $u \in [x_\infty, \overline{x}]$.
Writing $s=\min\{x,u\}$, 
the nonnegativity and monotonicity of $g$, 
$\xi \in [a,b]$, and $x_\infty\le r$  imply that $\xi g(s)\ge ag(x_\infty)\ge x_\infty$
and $\xi g(s)\le bg(\overline{x})=\overline{x}$. 
Hence
$F(\mathcal{X}\times\mathcal{U}\times\Xi)\subset\mathcal{X}$.

(iii)~Since $R$ and $g$ are continuously differentiable on the compact interval $[x_\infty,\overline{x}]$ and $(x,u)\mapsto\min\{x,u\}$ is $1$-Lipschitz, both $f$ and $F$ are Lipschitz continuous with a deterministic Lipschitz constant.
\end{proof}

This strict diminishing-marginal-effect property in \eqref{eq:reed-condition-A}
ensures that the base-stock-type target $x^\star$ from above 
is unique and \Cref{ass:uniqueness-type} holds, as shown next.

\begin{lemma}
\label{lem:harvest-redundant-targets}
There exists a unique escapement level \(x^\star\in[x_\infty,\overline x]\)
such that the constant target policy \(u^\star(x)\equiv x^\star\) is
optimal for the target-escapement reformulation. Equivalently, the
induced optimal harvest is \(h^\star(x)=[x-x^\star]_+\), and the
induced optimal post-harvest stock is \(\min\{x,x^\star\}\). Moreover,
for every state \(x\in\mathcal X\) and every optimal control
\(u\in\mathcal U^\star(x)\), we have
\(\min\{x,u\}=\min\{x,x^\star\}\). Consequently,
\Cref{ass:uniqueness-type} holds with
\begin{equation} 
\label{eq:harvest-assumption-2} 
F^\star(x,\xi)\coloneqq \xi g(\min\{x,x^\star\}), 
\qquad f^\star(x,\xi)\coloneqq R(\min\{x,x^\star\})-R(x). 
\end{equation}
\end{lemma}

\begin{proof}
For fixed \(x\in\mathcal X\), the feasible target controls are
\(u\in\mathcal U=[x_\infty,\overline x]\). Since
\(x\in[x_\infty,\overline x]\), the map
\(u\mapsto s=\min\{x,u\}\) has image \([x_\infty,x]\). Both the
one-stage cost and the dynamics depend on \(u\) only through this
post-harvest stock \(s\). Thus, after removing redundant controls, the
choice of a target \(u\) is equivalent to the choice of a post-harvest
stock \(s\in[x_\infty,x]\).

Following \cite{Reed}, we first derive the DP equation
for a transformation of the original reward maximization problem. 
Since $R'(s)=p-c(s)>0$ on $(x_\infty,\overline{x}]$, the map
$R:[x_\infty,\overline{x}]\to[0,M]$ is strictly increasing where $M\coloneqq R(\overline{x})$. 
Hence, its inverse \(R^{-1}\) exists. 
Moreover, it is continuously differentiable on \((0,M]\).
For $\xi \in [a,b]$, we define
$\mathsf{G}_\xi \colon [0, M] \to [0, M]$ by
\[
\mathsf{G}_\xi(v)\coloneqq R\bigl(\xi g(R^{-1}(v))\bigr).
\]
Let $W$  denote the normalized reward maximization value function in the transformed
stock variable (cf.\ eq.~(3.3) in \cite{Reed}), that is, 
\[
W(z)=\max_{0\le v\le z}
\bigl\{z-v+\gamma\mathbb{E}_{\xi\sim P}[W(\mathsf{G}_\xi(v))]\bigr\},
\qquad z\in[0,M].
\]
The proof of Theorem~1 in \cite{Reed}
shows that $W$ is continuous, concave, and nondecreasing.

We show that $W$ is strictly increasing. 
Let us define $\mathsf{H} \colon [0, M] \to \mathbb{R}$ by
\[
\mathsf{H}(v)\coloneqq -v+\gamma\mathbb{E}_{\xi\sim P}[W(\mathsf{G}_\xi(v))].
\]
If $0\le z_1<z_2\le M$ and
$v_1$ is optimal for $z_1$, then $v_1\le z_1<z_2$, so $v_1$ is feasible for
$z_2$. Therefore
\[
W(z_2)\ge
z_2 + \mathsf{H}(v_1)
=
W(z_1)+z_2-z_1>W(z_1).
\]
Thus \(W\) is strictly increasing.

Next, we show that the function $\mathsf{H}$
has a unique maximizer on $[0,M]$.
For $y\in(x_\infty,\overline{x}]$ and $v=R(y)$,
\[
\mathsf{G}_\xi'(R(y))
=\xi\frac{p-c(\xi g(y))}{p-c(y)}g'(y).
\]
Since \(\xi \geq a >0\), \(R\) is strictly increasing, the condition in
\eqref{eq:reed-condition-A}  implies that \(\mathsf{G}_\xi\) is nondecreasing and strictly concave
as a function of \(v\in(0,M]\). Therefore
\(\mathsf{G}_\xi\) is nondecreasing and strictly concave on \((0,M]\), and it
extends continuously to \([0,M]\).
Since \(\mathsf{G}_\xi\) is strictly concave on \((0,M]\) and continuous on
\([0,M]\), it is concave on \([0,M]\). Because \(W\) is concave and
nondecreasing, \(W\circ \mathsf{G}_\xi\) is concave on \([0,M]\). Moreover, since
\(W\) is strictly increasing and \(\mathsf{G}_\xi\) is strictly concave on
\((0,M]\), the composition \(W\circ \mathsf{G}_\xi\) is strictly concave on
\((0,M]\). Taking expectations preserves these properties, and adding the affine term
\(-v\) gives that \(\mathsf{H}\) is concave on \([0,M]\) and strictly concave on
\((0,M]\). Since \(\mathsf{H}\) is continuous on the compact interval \([0,M]\),
it has a maximizer; concavity on \([0,M]\) and strict concavity on
\((0,M]\) imply that this maximizer is unique. We denote it by
\(v^\star\in[0,M]\).

Fix \(x\in\mathcal X\). The feasible post-harvest stocks are
\(s\in[x_\infty,x]\). Since \(R\) is strictly increasing and
\(R(x_\infty)=0\), the change of variables \(v=R(s)\) maps this feasible
set onto \([0,R(x)]\).  Since \(\mathsf{H}\) is strictly concave
and has the unique maximizer \(v^\star\) on \([0,M]\), \(\mathsf{H}\) is
nondecreasing on \([0,v^\star]\) and nonincreasing on
\([v^\star,M]\). Hence its unique maximizer over \([0,R(x)]\) is
\(\min\{R(x),v^\star\}\).
Set
$x^\star\coloneqq R^{-1}(v^\star)$. Since $R$ is increasing, the unique optimal post-decision
stock is
\[
R^{-1}(\min\{R(x),v^\star\})=\min\{x,x^\star\}.
\]
Therefore every optimal control $u\in\mathcal U^\star(x)$ satisfies
$\min\{x,u\}=\min\{x,x^\star\}$. Substituting this identity into
\eqref{eq:harvest-fF} yields \eqref{eq:harvest-assumption-2}
for every optimal control. Hence \Cref{ass:uniqueness-type} holds.
\end{proof}

\subsection{Harvesting CLT}
 
We now verify the remaining hypotheses of \Cref{thm:inf-soc:clt} and obtain the corresponding central limit theorem.
 
\begin{proposition}[Harvesting  CLT]
\label{prop:harvest-clt}
For the target-escapement reformulation introduced above,
\[
  N^{1/2}(\hat V_N - V) \rightsquigarrow \mathfrak{G} \quad\text{in} \quad  C(\mathcal{X}),
\]
where $\mathfrak{G}$ is the unique centered Gaussian 
process on $\mathcal{X}$ satisfying
\[
  \mathfrak{G}(x) = \mathfrak{H}(x) + \gamma\,\mathbb{E}_{\xi \sim P}\bigl[\mathfrak{G}(\xi\,g(\min\{x,x^\star\}))\bigr], \qquad x \in \mathcal{X}.
\]
\end{proposition}
 
\begin{proof}
We verify the hypotheses of the CLT in \Cref{thm:inf-soc:clt}. We have already verified \Cref{inf-soc:ass-2,ass:uniqueness-type}; see \Cref{lem:harvest-assumption-1,lem:harvest-redundant-targets}.
 
Now, we verify the linearization condition in \eqref{eq:inf-soc:linearization-error} using \Cref{prop:main-linearization-conditions}. To this end, we verify the compactness condition of the operator $\mathcal{A}$ defined in \eqref{eq:A-operator}. The noise enters multiplicatively, so we pass to additive form: with $\zeta \coloneqq  \log\xi$, which has a density on $[\log a, \log b]$, and $G(x,u) \coloneqq  \log g(\min\{x,u\})$ continuous (as $g>0$ on $[x_\infty,\overline{x}]$), the operator $\mathcal{A}$ is given by
\[
  [\mathcal{A}W](x,u) = \mathbb{E}_{\xi \sim P}[W(\xi\,g(\min\{x,u\}))] = \mathbb{E}_{\zeta}[W(\exp(G(x,u)+\zeta))].
\]
Since $\zeta$ has a density, \Cref{lemma:additive-noise-transition-linearization} ensures compactness of $\mathcal{A}$.
 
Next, we verify the stochastic equicontinuity-type condition \eqref{eq:inf-soc:continuity} using \Cref{prop:main-equicontinuity-conditions,lem:Lipschitz-equicontinuity}. Since $f$ is bounded on the compact set $\mathcal{X}\times\mathcal{U}\times\Xi$, let $M_f \coloneqq  \|f\|_\infty$. The contraction property ensures $V$ and, with probability one, $\hat V_N$ have a sup norm of at most $(1-\gamma)^{-1}M_f$. 

To show that \(V\) and \(\hat V_N\) are nonincreasing, let
\(W\in C(\mathcal X)\). For fixed \(x\in\mathcal X\), the feasible
controls are \(u\in\mathcal U=[x_\infty,\overline x]\), and the
map \(u\mapsto \min\{x,u\}\) has image \([x_\infty,x]\). Since the
stage cost and dynamics depend on \(u\) only through the post-harvest
stock \(s\), the DP operator can be written as
\[
[\mathcal T W](x)
=
-R(x)
+
\min_{s\in[x_\infty,x]}
\left\{
R(s)+\gamma \mathbb{E}_{\xi \sim P}[W(\xi g(s))]
\right\}.
\]
The term \(-R(x)\) is nonincreasing because \(R\) is strictly
increasing. The minimum term is also nonincreasing in \(x\): if
\(x_1\le x_2\), then
\([x_\infty,x_1]\subseteq[x_\infty,x_2]\), and therefore
\[
\min_{s\in[x_\infty,x_2]}
\left\{
R(s)+\gamma\,\mathbb{E}_{\xi \sim P}[W(\xi g(s))]
\right\}
\le
\min_{s\in[x_\infty,x_1]}
\left\{
R(s)+\gamma\,\mathbb{E}_{\xi \sim P}[W(\xi g(s))]
\right\}.
\]
Hence \(\mathcal T W\) is nonincreasing for every
\(W\in C(\mathcal X)\).

The SAA DP operator has the same representation, with
the expectation replaced by the sample average. Hence
\(\hat{\mathcal T}_N W\) is also nonincreasing for every
\(W\in C(\mathcal X)\), with probability one. 
Therefore, with probability one,
\[
V,\hat V_N\in
\mathcal V
\coloneqq
\left\{
W\in C(\mathcal X):
\|W\|_\infty\le (1-\gamma)^{-1}M_f,
\quad
W \text{ is nonincreasing}
\right\}.
\]

Moreover, for each \((x,u)\in\mathcal X\times\mathcal U\), the map
\(\xi\mapsto F(x,u,\xi)=\xi g(\min\{x,u\})\) is nondecreasing on
\(\Xi=[a,b]\). Since \(\mathcal X\) and \(\Xi\) are intervals,
\Cref{lem:Lipschitz-equicontinuity}(i) yields
\eqref{eq:inf-soc:q-continuity}. Since \(\mathcal A\) is compact,
\Cref{prop:main-equicontinuity-conditions} implies
\eqref{eq:inf-soc:continuity}.
 
Having verified the hypotheses of \Cref{thm:inf-soc:clt}, we obtain the CLT.
\end{proof}

\section{Discussion}
\label{sec:discussion}

We developed a statistical limit theory for data-driven infinite-horizon SOC in which the estimator is defined by an empirical DP equation. If, for example, the optimal policy is unique, then the Gaussian limit law satisfies a linear fixed point equation, resembling the DP principle. For nonunique optimal policies, the limit law may be non-Gaussian but still solves a fixed point equation.
We have used the functional CLT to compute 
the variance of the SAA optimal value and compared it with the variance of the total discounted cost.
A simple example shows that these variances can be quite different, especially for discount factors close to one. 

Several extensions remain open. The i.i.d.\ noise assumption could be
relaxed to Markov-modulated uncertainty \cite{Chen2001,Milz2026,Sethi1997}. Verifying the hypotheses of
the CLT also remains challenging in general continuous-state
models; the sufficient conditions used here rely on compactness and
Donsker-type entropy arguments that are most transparent in low
dimension. Post-stage
or reduced-state formulations of stochastic dynamic programs
\cite{Pennanen2025,Powell2011}
could be useful to obtain CLTs for larger problem
classes.
It would also be interesting to construct confidence intervals for optimal values, for example using subsampling methods along the lines of those studied for two-stage stochastic programming in \cite{Eichhorn2007}.

\section*{Acknowledgments}

We are grateful to Prof.\ Alexander Shapiro for many insightful discussions. 
The authors acknowledge the use of ChatGPT 5.5 and 5.6, and Claude (Opus 4.8) for assistance with language editing, organization, and
presentation.  
The authors reviewed and revised
all generated text and take responsibility for the final content.

\appendix

\section{Proofs for \Cref{sec-basic}}

\subsection{Proof of \Cref{thm:inf-soc:clt}}
\label{app:main-clt-proofs}

We prepare our proof of \Cref{thm:inf-soc:clt}.
While value functions are generally
not Lipschitz continuous
(see, e.g., Example~1 in \cite{Wang2026}),
we show that they are 
H\"older continuous.

\begin{lemma}[H\"older continuity of population value function]
\label{lem:inf-soc:holder-value}
If \Cref{inf-soc:ass-2} holds, 
then for all sufficiently small \(\alpha\in(0,1]\), 
$V$ is H\"older continuous with exponent $\alpha$.
\end{lemma}

\begin{proof}
Let 
$\operatorname{diam}(\mathcal{X})
\coloneqq 
\sup_{x,x' \in \mathcal{X}} \|x-x'\|_2$.
If \(\operatorname{diam}(\mathcal X)=0\), the claim 
holds. Set
\(D\coloneqq \operatorname{diam}(\mathcal X)>0\).
We first choose the H\"older exponent. For \(\alpha\in(0,1]\), we have
\(K(\xi)^\alpha\le 1+K(\xi)^2\). Moreover,
\(K(\xi)^\alpha\to\mathbf 1_{\{K(\xi)>0\}}\) as \(\alpha\downarrow0\).
Thus, by dominated convergence,
$
\mathbb{E}_{\xi\sim P}[K(\xi)^\alpha]
\to
\mathrm{Prob}\{K(\xi)>0\}
$.
Since \(\gamma<1\), we can choose \(\alpha\in(0,1]\) such that
\(\beta\coloneqq\gamma\mathbb{E}_{\xi\sim P}[K(\xi)^\alpha]<1\).

Let \(W\in C(\mathcal X)\) be \(\alpha\)-H\"older continuous, with
seminorm \([W]_\alpha\). For \(x,x'\in\mathcal X\),
\[
\begin{aligned}
|(\mathcal TW)(x)-(\mathcal TW)(x')|
&\le
\sup_{u\in\mathcal U}
\mathbb{E}\Big[
|f(x,u,\xi)-f(x',u,\xi)|  \\
&\qquad\qquad
+\gamma |W(F(x,u,\xi))-W(F(x',u,\xi))|
\Big]  \\
&\le
\mathbb{E}_{\xi\sim P}[K(\xi)]\|x-x'\|_2
+\gamma[W]_\alpha\mathbb{E}_{\xi\sim P}[K(\xi)^\alpha]\|x-x'\|_2^\alpha .
\end{aligned}
\]
Since \(\|x-x'\|_2\le D\), we have
\(\|x-x'\|_2\le D^{1-\alpha}\|x-x'\|_2^\alpha\). Hence
\[
[\mathcal TW]_\alpha
\le
D^{1-\alpha}\mathbb{E}_{\xi\sim P}[K(\xi)]
+
\beta[W]_\alpha .
\]

Now let \(V_0\coloneqq0\) and \(V_{k+1}\coloneqq\mathcal TV_k\).
Writing \(L_k\coloneqq[V_k]_\alpha\), the preceding bound gives
\(L_{k+1}\le A+\beta L_k\), where
\(A\coloneqq D^{1-\alpha}\mathbb{E}_{\xi\sim P}[K(\xi)]\). Since \(L_0=0\),
\[
L_k
\le
A\sum_{j=0}^{k-1}\beta^j
\le
\frac{A}{1-\beta}.
\]
Thus, for all \(k\ge1\),
\[
|V_k(x)-V_k(x')|
\le
\frac{D^{1-\alpha}\mathbb{E}_{\xi\sim P}[K(\xi)]}
{1-\gamma\mathbb{E}_{\xi\sim P}[K(\xi)^\alpha]}
\|x-x'\|_2^\alpha .
\]
Since \(\mathcal T\) is a \(\gamma\)-contraction,
\(V_k\to V\) uniformly. Passing to the limit gives
the H\"older continuity.
\end{proof}

We recall the notion of Hadamard directional
differentiability tangentially to a set; see p.~461 in \cite{SDR}.
Let \(\mathbb B_1\) and \(\mathbb B_2\) be Banach spaces, let
\(\mathbb{K} \subseteq\mathbb B_1\), let \(G:\mathbb 
B_1 \to\mathbb B_2\), and let
\(\mu\in \mathbb{K}\). We say that \(G\) is Hadamard directionally
differentiable at \(\mu\) tangentially to \(\mathbb{K}\) if there exists a map
\(G'(\mu; \cdot)\) such that, for every \(t_N\downarrow0\) and every \(s_N\to s\)
such that \(\mu+t_Ns_N\in \mathbb{K}\),
\[
t_N^{-1}\big(G(\mu+t_Ns_N)-G(\mu)\big)
\to
G'(\mu; s)
\quad
\text{in} \quad  \mathbb B_2 .
\]
If $\mathbb{K} = \mathbb{B}_1$,
then $G$ is said to be
Hadamard directionally differentiable
at $\mu$.
If \(\mathbb{K}\) is a closed linear subspace of \(\mathbb B_1\), and \(G'(\mu; \cdot)\)
is linear and continuous on it, then \(G\) is Hadamard differentiable
at \(\mu\) tangentially to \(\mathbb{K}\).

We recall the definition of 
the inf operator $\Psi$ from 
\eqref{eq:inf-operator}.
We define the solution set
\begin{align}
\label{eq:inf-operator-solution-set}
\mathscr{U}^\star(q;x)
\coloneqq 
\operatorname*{argmin}_{u\in\mathcal U}q(x,u).
\end{align}
For all $x \in \mathcal{X}$, 
we have
$\mathcal{U}^\star(x) = \mathscr{U}^\star(\mathcal Q(V);x)$,
where $\mathcal{U}^\star$ is the solution set
from \eqref{eq:solutionset}.
For \(q\in C(\mathcal X\times\mathcal U)\), we define
\begin{align}
S^\star(q)
\coloneqq
\left\{
h\in C(\mathcal X\times\mathcal U):
h(x,u)=h(x,v)
\text{ whenever }
u,v\in\mathscr{U}^\star(q;x)
\right\}.
\end{align}
This closed linear subspace consists of  directions that do
not distinguish between optimal controls at the same state.
We also recall the definition of 
 \(\mathcal L\) from \eqref{eq:L}.

\begin{lemma}
\label{lem:infimum-map-derivative}
If 
\Cref{inf-soc:ass-2}\textup{(i)} holds,
and
\(q\in C(\mathcal X\times\mathcal U)\),
then the  following holds.
\begin{enumerate}[label=\textup{(\roman*)},nosep]
    \item 
The map \(\Psi\) is Hadamard  differentiable at \(q\)
tangentially to \(S^\star(q)\), with 
\[
[\Psi'(q;h)](x)
=
h(x,u),
\qquad
u\in\mathscr{U}^\star(q;x),
\quad h\in S^\star(q).
\]
\item  If
\(\mathscr{U}^\star(q;\cdot)\)
is lower hemicontinuous, then $\Psi$ is
Hadamard directionally differentiable at \(q\), with 
\[
[\Psi'(q;h)](x)
=
\min_{u\in\mathscr{U}^\star(q;x)}h(x,u).
\]
\item If, in addition,
\Cref{inf-soc:ass-2,ass:uniqueness-type} hold, then
\(\mathcal T\) is Hadamard
differentiable at \(V\), with
$
\mathcal T'(V;W)=\gamma\mathcal LW
$, 
$W\in C(\mathcal X)$.
\end{enumerate}
\end{lemma}

\begin{proof}
\textup{(i)}
We first note that \(q\in S^\star(q)\). 
Since \(S^\star(q)\) is a closed linear subspace of
\(C(\mathcal X\times\mathcal U)\), the tangential
Hadamard differentiability statement may equivalently be viewed as
Hadamard differentiability of the restriction of \(\Psi\) to
\(S^\star(q)\) at \(q\).
For all $h\in S^\star(q)$, and
$u\in\mathscr{U}^\star(q;x)$, we have
\begin{align*}
    h(x,u)
    = 
    \min_{v\in\mathscr{U}^\star(q;x)}
    \, h(x,v)
    = 
    \max_{v\in\mathscr{U}^\star(q;x)}
    \, h(x,v).
\end{align*}
Applying Theorem~5 in \cite{Hogan1973} to the minimum and maximum shows
that this common value is continuous in \(x\).
Since $\Psi$ is Lipschitz,
it suffices to show directional
differentiability;
see Proposition~2.49 
in \cite{BS2000}.
Danskin's lemma implies
the derivative expressions
for each fixed $x \in \mathcal{X}$.
Now, the pointwise convergence
can be upgraded to uniform convergence
using  
Dini's theorem.

\textup{(ii)}
Theorem~7 in \cite{Hogan1973}
ensures continuity of
$x\mapsto\min_{v\in\mathscr{U}^\star(q;x)}
    \, h(x,v)$.
Now, the verification is similar to 
that in part~(i).

\textup{(iii)}
Recall from
\eqref{eq:Q}, and
\eqref{eq:inf-operator} that
\(\mathcal T=\Psi\circ\mathcal Q\). The map
\(\mathcal Q:C(\mathcal X)\to C(\mathcal X\times\mathcal U)\)
is affine and Hadamard differentiable at \(V\), with
\[
[\mathcal Q'(V;W)](x,u)
=
\gamma\mathbb{E}_{\xi\sim P}
\left[
W(F(x,u,\xi))
\right]
=
\gamma[\mathcal AW](x,u),
\]
where \(\mathcal A\) is defined in \eqref{eq:A-operator}.
If
\(u,v\in\mathscr U^\star(\mathcal Q(V);x)=\mathcal U^\star(x)\),
where the solution sets are defined in
\eqref{eq:inf-operator-solution-set} and \eqref{eq:solutionset},
then \Cref{ass:uniqueness-type} gives
for all $\xi \in \Xi$, 
\[
F(x,u,\xi)=F(x,v,\xi)=F^\star(x,\xi).
\]
Hence
\(\mathcal Q'(V;W)\in S^\star(\mathcal Q(V))\) for every
\(W\in C(\mathcal X)\).
By part~\textup{(i)} and the chain rule for Hadamard
differentiability, for \(u\in\mathcal U^\star(x)\),
\[
[\mathcal T'(V;W)](x)
=
[\mathcal Q'(V;W)](x,u)
=
\gamma\mathbb{E}_{\xi\sim P}
\left[
W(F^\star(x,\xi))
\right]
=
\gamma[\mathcal LW](x).
\]
\end{proof}

\begin{lemma}
\label{lem:statistical-limit}
Suppose that \Cref{inf-soc:ass-2} holds. 
Then the following holds.
\begin{enumerate}[label=\textup{(\roman*)},nosep]
    \item The CLT in \eqref{eq:Q-CLT} holds, that is, 
\[
N^{1/2}(\hat{\mathcal Q}_N(V)-\mathcal Q(V))
\rightsquigarrow \mathfrak Y
\quad
\text{in}
\quad C(\mathcal X\times\mathcal U),
\]
where \(\mathfrak Y\) is a centered Gaussian random element in
\(C(\mathcal X\times\mathcal U)\) with covariance kernel
\[
\operatorname{Cov}
\bigl(
\mathfrak Y(x,u),\mathfrak Y(x',u')
\bigr)
=
\operatorname{Cov}_{\xi\sim P}
\bigl(
\Phi(x,u,\xi),\Phi(x',u',\xi)
\bigr).
\]

\item If \Cref{ass:uniqueness-type} holds, then, 
$
\mathcal Q(V)
\in S^\star(\mathcal Q(V))
$,
and
with probability one,
$$
\hat{\mathcal Q}_N(V),\,
\hat{\mathcal Q}_{N}(\hat V_N),\,
\mathcal Q(\hat V_N), \,
\mathfrak Y
\in S^\star(\mathcal Q(V)).
$$
\item If either \Cref{ass:uniqueness-type}
holds, or
$\mathcal{U}^\star$ is lower hemicontinuous, then,
in $C(\mathcal X)$,
\begin{align}
\label{eq:application-delta-theorem}
N^{1/2}\bigl(\hat{\mathcal T}_N V-\mathcal T V\bigr)
=
N^{1/2}
\bigl(
\Psi(\hat{\mathcal Q}_N(V))-\Psi(\mathcal Q(V))
\bigr)    \rightsquigarrow
\Psi'(\mathcal Q(V);\mathfrak Y).
\end{align}
Moreover, under \Cref{ass:uniqueness-type},
the CLT in 
\eqref{eq:inf-soc:clt-mathcalV} holds with
\[
\mathfrak H(x)
\coloneqq 
[\Psi'(\mathcal Q(V);\mathfrak Y)](x)
=
\mathfrak Y(x,\pi^\star(x)),
\qquad x\in\mathcal X,
\]
where \(\pi^\star\) is any measurable selector of
\(\mathcal{U}^\star(x)\). 
\end{enumerate}
\end{lemma}

Our proof of \Cref{lem:statistical-limit} 
presented below requires further notation and notions.
For a set $E$,
let $\ell^\infty(E)$ denote the space of bounded
real-valued functions on \(E\), equipped with 
the sup norm. 
A class \(\mathcal H\) of measurable real-valued functions on 
$\Xi$ is \(P\)-Donsker if the empirical process
\[
    \bigg\{
    N^{-1/2}\sum_{i=1}^N
    \bigl(h(\xi^{(i)})-\mathbb{E}_{\xi \sim P}[h(\xi)]\bigr)
    : h\in\mathcal H
    \bigg\}
\]
converges in distribution in \(\ell^\infty(\mathcal H)\) to a tight Borel random element; see Section~2.1 in \cite{Vaart2023}.

\begin{proof}[{Proof of \Cref{lem:statistical-limit}}]
\textup{(i)} We show that the function class
$
\mathcal G
\coloneqq
\{\Phi(x,u,\cdot):(x,u)\in \mathcal X\times\mathcal U\}
$
is $P$-Donsker,
where $\Phi$ is defined in \eqref{eq:Phi}.
By \Cref{lem:inf-soc:holder-value}, there exist
\(\alpha\in(0,1]\) and \(L_V>0\) such that
\[
|V(y)-V(y')|
\le
L_V\|y-y'\|^\alpha,
\qquad
y,y'\in\mathcal X.
\]
For \((x,u),(x',u')\in\mathcal X\times\mathcal U\), set
\[
\delta
\coloneqq
\|x-x'\|_2+\|u-u'\|_2.
\]
Using \Cref{inf-soc:ass-2}(iii), we obtain for every \(\xi\in\Xi\),
\[
\begin{aligned}
|\Phi(x,u,\xi)-\Phi(x',u',\xi)|
&\le
K(\xi)\delta
+
\gamma L_V
\|F(x,u,\xi)-F(x',u',\xi)\|^\alpha
\\
&\le
K(\xi)\delta
+
\gamma L_V K(\xi)^\alpha \delta^\alpha.
\end{aligned}
\]
Since \(\mathcal X\) and \(\mathcal U\) are compact, there exists
\(D>0\) such that \(\delta\le D\) for all
\((x,u),(x',u')\in\mathcal X\times\mathcal U\). Therefore
\[
\delta \le D^{1-\alpha}\delta^\alpha,
\]
and hence
$
|\Phi(x,u,\xi)-\Phi(x',u',\xi)|
\le
M(\xi)\delta^\alpha
$,
where
\[
M(\xi)
\coloneqq
D^{1-\alpha}K(\xi)+\gamma L_V K(\xi)^\alpha.
\]
Because \(\mathbb{E}_{\xi \sim P}[K(\xi)^2]<\infty\), we have
\(\mathbb{E}_{\xi \sim P}[M(\xi)^2]<\infty\). 
Since the set 
\(\mathcal X\times\mathcal U \subset \mathbb{R}^n \times \mathbb{R}^m\) is compact, 
it follows from the standard entropy criterion for Donsker classes
(see Theorems~2.5.6 and 2.7.17 in \cite{Vaart2023})
that \(\mathcal G\) is \(P\)-Donsker. Consequently,
\[
N^{1/2}(\hat{\mathcal{Q}}_N(V) - \mathcal{Q}(V))
\rightsquigarrow
\mathfrak Y
\quad\text{in}\quad
\ell^\infty(\mathcal X\times\mathcal U),
\]
where \(\mathfrak Y\) is a centered Gaussian process with covariance
function
\[
\mathrm{Cov}\bigl(
\mathfrak Y(x,u),\mathfrak Y(x',u')
\bigr)
=
\mathrm{Cov}\bigl(
\Phi(x,u,\xi),\Phi(x',u',\xi)
\bigr).
\]

We now upgrade this to  convergence in distribution in
\(C(\mathcal X\times\mathcal U)\). Each
\(N^{1/2}(\hat{\mathcal Q}_N(V)-\mathcal Q(V))\) belongs to
\(C(\mathcal X\times\mathcal U)\). Since
\(C(\mathcal X\times\mathcal U)\) is closed in
\(\ell^\infty(\mathcal X\times\mathcal U)\), the Portmanteau theorem
implies that the distributional limit is supported on
\(C(\mathcal X\times\mathcal U)\). Therefore, after identifying the
limit with its \(C(\mathcal X\times\mathcal U)\)-valued version,
Theorem~1.3.10 in \cite{Vaart2023} yields
\[
N^{1/2}(\hat{\mathcal Q}_N(V)-\mathcal Q(V))
\rightsquigarrow
\mathfrak Y
\quad\text{in} \quad C(\mathcal X\times\mathcal U).
\]

\textup{(ii)}
Now, let \Cref{ass:uniqueness-type} hold. 
If
\(u,v\in\mathscr{U}^\star(\mathcal{Q}(V);x)\), then
for all $\xi \in \Xi$, 
\[
f(x,u,\xi)=f(x,v,\xi) = f^\star(x,\xi),
\quad
\text{and}
\quad 
F(x,u,\xi)=F(x,v,\xi) = F^\star(x,\xi),
\]
and for all $W \in C(\mathcal{X})$,
\begin{align*}
f(x,u,\xi) + \gamma W\bigl(F(x,u,\xi)\bigr)
=
f(x,v,\xi) + \gamma W\bigl(F(x,v,\xi)\bigr)
=
f^\star(x,\xi) + \gamma W\bigl(F^\star(x,\xi)\bigr).
\end{align*}
Applying (empirical) expectations,
we have
$\mathcal{Q}(V) \in S^\star(\mathcal{Q}(V))$,
and
with probability one,
$\hat{\mathcal Q}_N(V), \hat{\mathcal Q}_{N}(\hat V_N) \in S^\star(\mathcal{Q}(V))$.
Since \(S^\star(\mathcal{Q}(V))\) is a
closed linear subspace of \(C(\mathcal X\times\mathcal U)\), the distributional
limit $\mathfrak Y$ also satisfies
$\mathfrak Y\in S^\star(\mathcal{Q}(V))$.

\textup{(iii)}
If \Cref{ass:uniqueness-type} holds, then
\Cref{lem:infimum-map-derivative} and part~(ii) imply that
\(\Psi\) is Hadamard differentiable at \(\mathcal{Q}(V)\) tangentially to \(S^\star(\mathcal{Q}(V))\).
If, on the other hand, 
$\mathcal{U}^\star$ is lower hemicontinuous, then
\Cref{lem:infimum-map-derivative} implies that $\Psi$
is Hadamard directionally differentiable. 
In either case, the delta method gives
\[
N^{1/2}
\big(
\hat{\mathcal T}_N V
-
\mathcal T V
\big)
=
N^{1/2}
\bigl(
\Psi(\hat{\mathcal Q}_N(V))-\Psi(\mathcal{Q}(V))
\bigr)
\rightsquigarrow
\Psi'(\mathcal{Q}(V);\mathfrak Y)
\quad
\text{in} \quad  C(\mathcal X).
\]
This proves \eqref{eq:application-delta-theorem}.
Under \Cref{ass:uniqueness-type},
\Cref{lem:infimum-map-derivative} ensures, for \(h\in S^\star(\mathcal{Q}(V))\),
\[
[\Psi'(\mathcal{Q}(V);h)](x)=h(x,u),
\qquad
u\in\mathcal{U}^\star(x)=\mathscr{U}^\star(\mathcal{Q}(V);x).
\]
 Taking any selector
\(\pi^\star(x)\in\mathcal{U}^\star(x)\), we obtain
\[
[\Psi'(\mathcal{Q}(V);\mathfrak Y)](x)
=
\mathfrak Y(x,\pi^\star(x)).
\]
Because \(\mathfrak Y\in S^\star(\mathcal{Q}(V))\) with probability one, this value is
independent of the chosen selector. Thus
\(\mathfrak H=\Psi'(\mathcal{Q}(V);\mathfrak Y)\), which proves the claimed CLT in
\eqref{eq:inf-soc:clt-mathcalV}.
\end{proof}

\begin{proof}[{Proof of \Cref{thm:inf-soc:clt}}]
\Cref{lem:statistical-limit} ensures the CLT in
\eqref{eq:inf-soc:clt-mathcalV}.
From the DP equations in \eqref{eq:bellman-fixed-points}, 
\[
\begin{aligned}
N^{1/2}(\hat V_N-V)
&=
N^{1/2}(\hat{\mathcal T}_NV-\mathcal TV)
+
\gamma\mathcal L
N^{1/2}(\hat V_N-V)       \\
&\quad
+
N^{1/2}
\bigl(
\hat{\mathcal T}_N\hat V_N
-
\hat{\mathcal T}_NV
-
\gamma\mathcal L(\hat V_N-V)
\bigr).
\end{aligned}
\]
By \eqref{eq:inf-soc:fixed-point-linearization}, the last term is
\(o_p(1)\) in \(C(\mathcal X)\). Therefore
\[
(I-\gamma\mathcal L)N^{1/2}(\hat V_N-V)
=
N^{1/2}(\hat{\mathcal T}_NV-\mathcal TV)
+
o_p(1).
\]
Since \(I-\gamma\mathcal L\) is continuously invertible,
Slutsky's theorem, and the continuous mapping theorem imply
\[
N^{1/2}(\hat V_N-V)
\rightsquigarrow
(I-\gamma\mathcal L)^{-1}\mathfrak H
\quad
\text{in}
\quad C(\mathcal X).
\]
\end{proof}

\section{Proofs of \Cref{prop:main-linearization-conditions,lemma:additive-noise-transition-linearization,prop:main-equicontinuity-conditions,lem:Lipschitz-equicontinuity}}
\label{subsect:proofs-sufficient-conditions}

We prepare our proof of  \Cref{prop:main-linearization-conditions}.
The following estimate is repeatedly used throughout the appendices.

\begin{lemma}
If \Cref{inf-soc:ass-2} holds,
then
\begin{align}
\label{eq:scaled-value-function-difference-bound}
N^{1/2}\|\hat{V}_N - V\|_{\infty} = O_p(1).
\end{align}
\end{lemma}
\begin{proof}
Since \(\mathcal T\) and \(\hat{\mathcal T}_N\) are
\(\gamma\)-contractions
(see, e.g.,
eq.\ (4.2) in \cite{Shapiro2020}),
we have
$
\|\hat{V}_N - V\|_\infty
\leq (1-\gamma)^{-1}
\|\hat{\mathcal T}_NV-\mathcal TV\|_\infty
$.
Moreover, since the infimum map is nonexpansive,
\[
\|\hat{\mathcal T}_NV-\mathcal TV\|_\infty
\le
\|\hat{\mathcal Q}_N(V)-\mathcal Q(V)\|_\infty.
\]
Now, the CLT in \Cref{lem:statistical-limit}\textup{(i)} implies \eqref{eq:scaled-value-function-difference-bound}.
\end{proof}

For our next proof, we require the notation of tightness of a sequence of random elements.
Let \(Y_N\) be a  random elements 
in $C(\mathcal{Y})$,
with $\mathcal{Y}$ compact.
The sequence \((Y_N)\) is said to be tight if, for every
\(\varepsilon>0\), there exists a compact set
\(\mathcal{Y}_\varepsilon\subset C(\mathcal{Y})\) such that
\[
\mathrm{Prob}\{Y_N\in \mathcal{Y}_\varepsilon\}
\geq 1-\varepsilon
\quad \text{for all} \quad N \in \mathbb{N}.
\]
Equivalently, every subsequence of \((Y_N)\) admits a further subsequence that converges in distribution in \(C(\mathcal{Y})\); see Theorems~5.1 and~5.2 in \cite{Billingsley1999}.

\begin{proof}[{Proof of \Cref{prop:main-linearization-conditions}}]
\textup{(i)}
Using \eqref{eq:scaled-value-function-difference-bound}, 
we have
$\hat{V}_N - V\to 0$ as $N \to \infty$
in probability.
Combined with the Hadamard differentiability
established in \Cref{lem:infimum-map-derivative}, 
the local expansion \eqref{eq:frechet-linearization-error}, 
and \Cref{lem:statistical-limit}, 
we obtain 
\eqref{eq:inf-soc:linearization-error}.

\textup{(ii)}
Let us define 
$
Y_N \coloneqq N^{1/2}(\hat{V}_N-V)
$.
The bound \eqref{eq:scaled-value-function-difference-bound} ensures
$\|Y_N\|_\infty = O_p(1)$. 
Since $\mathcal{A}$ is compact, 
$(\mathcal{A}Y_N)$ is tight. 
We show that every subsequence of the left-hand side of
\eqref{eq:inf-soc:linearization-error} has a further subsequence
converging to zero in probability, without making the subsequences
explicit. By tightness, any subsequence of $(\mathcal{A}Y_N)$
has a further  subsequence, again denoted by
$(\mathcal{A}Y_N)$, converging in distribution to some random element in 
 $C(\mathcal{X}\times\mathcal{U})$.
Using the definition of $\mathcal{Q}$, we have
\begin{align*}
    N^{1/2}
    \big(\mathcal{Q}(\hat{V}_N) - \mathcal{Q}(V) \big)
    = \gamma \mathcal{A}Y_N.
\end{align*}
Combined with \Cref{lem:statistical-limit,lem:infimum-map-derivative} and 
the delta theorem
(see Theorem~9.76 in \cite{SDR}), 
$\mathcal{T}V = \Psi(\mathcal{Q}(V))$, 
and
$\mathcal{T}\hat{V}_N = \Psi(\mathcal{Q}(\hat{V}_N))$, 
we obtain
\eqref{eq:inf-soc:linearization-error}.

\textup{(iii)}
If \(\mathcal U\) is a singleton, then the minimizer is fixed and
$
\mathcal T\hat V_N-\mathcal T V
=
\gamma\mathcal L(\hat V_N-V)
$.
Thus, we obtain the assertion. 
Now let $\mathcal{U}$ contain at least two elements.
Since \(\mathcal U\) is finite and \([\mathcal{Q}(V)](\cdot,u)\) is continuous for
each \(u\), the pointwise uniqueness of the minimizer implies the
uniform gap
\begin{align}
\label{eq:optimality-gap}
\Delta
\coloneqq
\inf_{x\in\mathcal X}
\min_{\substack{u\in\mathcal U\\ u\neq\pi^\star(x)}}
\left\{
[\mathcal{Q}(V)](x,u)-[\mathcal{Q}(V)](x,\pi^\star(x))
\right\}
>0 .
\end{align}

Using \eqref{eq:scaled-value-function-difference-bound}, 
we have
$
\|\hat V_N-V\|_\infty=o_p(1)
$.
Define
\[
\mathcal E_N
\coloneqq
\left\{
2\gamma\|\hat V_N-V\|_\infty<\Delta
\right\}.
\]
Then \(\mathrm{Prob}\{\mathcal E_N\}\to1\). On \(\mathcal E_N\), for every
\(x\in\mathcal X\) and every \(u\neq\pi^\star(x)\),
\[
\begin{aligned}
[\mathcal{Q}(\hat V_N)](x,u)
-
[\mathcal{Q}(\hat V_N)](x,\pi^\star(x))
&
\ge
[\mathcal{Q}(V)](x,u)-[\mathcal{Q}(V)](x,\pi^\star(x))
-
2\gamma\|\hat V_N-V\|_\infty
\\
&
\ge
\Delta-2\gamma\|\hat V_N-V\|_\infty
>0 .
\end{aligned}
\]
Hence, on \(\mathcal E_N\), the minimizer in
\((\mathcal T\hat V_N)(x)\) is \(\pi^\star(x)\) for every
\(x\in\mathcal X\). Therefore, on \(\mathcal E_N\),
\[
\begin{aligned}
(\mathcal T\hat V_N)(x)-(\mathcal T V)(x)
&=
[\mathcal{Q}(\hat V_N)](x,\pi^\star(x))
-
[\mathcal{Q}(V)](x,\pi^\star(x))
\\
&=
\gamma[\mathcal L(\hat V_N-V)](x).
\end{aligned}
\]
Combined with \(\mathrm{Prob}\{\mathcal E_N\}\to1\),
we obtain \eqref{eq:inf-soc:linearization-error}.
\end{proof}

\begin{proof}[{Proof of \Cref{lemma:additive-noise-transition-linearization}}]
Let \(\rho: \Xi \to[0,\infty)\) be the density of $\xi$,
which   by zero outside
of $\Xi$.
We define \(\mathcal{Z} \coloneqq G(\mathcal X\times\mathcal U)\),
which is compact. We also define 
$\mathcal B:C(\mathcal X)\to C(\mathcal Z)$
by 
\[
[\mathcal{B}W](z)
\coloneqq 
\int_{\mathcal{Z} + \Xi} W(\kappa(y))\rho(y-z)\,\mathrm dy .
\]
Then \([\mathcal A W](x,u)=[\mathcal{B}W](G(x,u))\). 
We show that the image of $\mathcal{B}$ under 
the closed unit ball $\{W\in C(\mathcal X):\|W\|_\infty\le1\}$ is precompact. 
Since
\(|[\mathcal{B}W](z)|\le\|W\|_\infty
\), 
this image is bounded. 
Moreover, for
\(z,z'\in \mathcal{Z}\) and \(W\in C(\mathcal{X})\)
with $\|W\|_\infty \leq 1$, we have 
\[
\begin{aligned}
|[\mathcal{B}W](z)-[\mathcal{B}W](z')|
&\le
\int_{\mathcal{Z} + \Xi}
|W(\kappa(y))|\,|\rho(y-z)-\rho(y-z')|\,\mathrm dy  \\
&\le
\int_{\mathcal{Z} + \Xi}
|\rho(y-z)-\rho(y-z')|\,\mathrm dy .
\end{aligned}
\]
The right-hand side converges to zero as \(z'\to z\)
independently of $W$, by continuity of
translations;
see Lemma~4.3 in \cite{Brezis2011}. Hence,
the image of $\mathcal{B}$ over the closed unit ball in $C(\mathcal{X})$ is uniformly equicontinuous. 
By the Arzel\`a--Ascoli theorem
(see, e.g., Theorem~4.25 in \cite{Brezis2011}),
\(\mathcal{B}\) is compact.
Since the composition map \(h\mapsto h\circ G\) is bounded and linear
from \(C(\mathcal{Z})\) to \(C(\mathcal X\times\mathcal U)\), it follows that
\(\mathcal A\) is compact.
\end{proof}

\begin{proof}[{Proof of \Cref{prop:main-equicontinuity-conditions}}]
\textup{(i)}
Using the definition of $\Psi$
from \eqref{eq:inf-operator}, we have
$
\mathcal T=\Psi\circ\mathcal Q,
$
and 
$
\hat{\mathcal T}_N=\Psi\circ\hat{\mathcal Q}_N
$.
Hence
\begin{align*}
\big[\hat{\mathcal{T}}_{N}
-
\mathcal{T}\big](\hat{V}_{N})
-
\big[
\hat{\mathcal{T}}_{N}- \mathcal{T}
\big]
(V)
& =
\Psi\bigl(\hat{\mathcal Q}_{N}(\hat V_N)\bigr)
-
\Psi\bigl(\mathcal{Q}(V)\bigr)
\\
&\quad
-
\big[
\Psi\bigl(\mathcal{Q}(\hat V_N)\bigr)
-
\Psi\bigl(\mathcal{Q}(V)\bigr)
\big]
\\
&\quad
-
\big[
\Psi\bigl(\hat{\mathcal Q}_{N}(V)\bigr)
-
\Psi\bigl(\mathcal{Q}(V)\bigr)
\big].
\end{align*}

\Cref{lem:statistical-limit} ensures
that $N^{1/2}(\hat{\mathcal{Q}}_N(V) - \mathcal{Q}(V))$
converges in distribution.
Since $\mathcal{A}$ is compact,
the proof of \Cref{prop:main-linearization-conditions}
shows that $N^{1/2}(\mathcal{Q}(\hat V_N) - \mathcal{Q}(V))$
is tight. Combined with \eqref{eq:inf-soc:q-continuity}, and
\begin{align*}
\hat{\mathcal Q}_N(\hat V_N)-
\mathcal{Q}(V)
&=
\Big[\big[\hat{\mathcal Q}_N-\mathcal Q\big](\hat V_N)
-
\big[\hat{\mathcal Q}_N-\mathcal Q\big](V)
\Big]
\\
&\quad
+
\big[\hat{\mathcal Q}_N(V)-\mathcal Q(V) \big]
+
\big[\mathcal Q(\hat V_N)-\mathcal Q(V)\big],
\end{align*}
we have 
$N^{1/2}(\hat{\mathcal Q}_{N}(\hat V_N)
-
\mathcal{Q}(V))$
is tight.

By \Cref{lem:infimum-map-derivative} and
\Cref{inf-soc:ass-2,ass:uniqueness-type}, the map \(\Psi\) is Hadamard
differentiable at \(\mathcal Q(V)\)
tangentially to $S^\star(\mathcal{Q}(V))$.
\Cref{lem:statistical-limit} yields
$\mathcal Q(V) \in S^\star(\mathcal{Q}(V))$
and, 
with probability one,
$\hat{\mathcal Q}_{N}(\hat V_N)$,
$\mathcal Q(\hat V_N)$,
$\hat{\mathcal Q}_{N}(V) \in S^\star(\mathcal{Q}(V))$.

The preceding tightness statements imply that, along every subsequence,
there is a further subsequence along which
$
N^{1/2}\bigl(\hat{\mathcal Q}_N(\hat V_N)-\mathcal Q(V)\bigr)
$,
$
N^{1/2}\bigl(\mathcal Q(\hat V_N)-\mathcal Q(V)\bigr)
$,
and
$
N^{1/2}\bigl(\hat{\mathcal Q}_N(V)-\mathcal Q(V)\bigr)
$
converge jointly in distribution in
\(S^\star(\mathcal{Q}(V))^3\). Along this further subsequence, the
functional delta theorem gives, in \(C(\mathcal X)\),
\[
\begin{aligned}
\Psi\bigl(\hat{\mathcal Q}_{N}(\hat V_N)\bigr)
-
\Psi\bigl(\mathcal Q(V)\bigr) &=
\Psi'\bigl(
\mathcal Q(V);
\hat{\mathcal Q}_{N}(\hat V_N)-\mathcal Q(V)
\bigr)
+o_p(N^{-1/2}),
\\
\Psi\bigl(\mathcal Q(\hat V_N)\bigr)
-
\Psi\bigl(\mathcal Q(V)\bigr) &=
\Psi'\bigl(
\mathcal Q(V);
\mathcal Q(\hat V_N)-\mathcal Q(V)
\bigr)
+o_p(N^{-1/2}),
\\
\Psi\bigl(\hat{\mathcal Q}_{N}(V)\bigr)
-
\Psi\bigl(\mathcal Q(V)\bigr) &=
\Psi'\bigl(
\mathcal Q(V);
\hat{\mathcal Q}_{N}(V)-\mathcal Q(V)
\bigr)
+o_p(N^{-1/2}).
\end{aligned}
\]
Combining these three expansions with the preceding decomposition, and
using the linearity of \(\Psi'(\mathcal Q(V);\cdot)\), yields
\[
\begin{aligned}
&\big[\hat{\mathcal T}_{N}-\mathcal T\big](\hat V_N)
-
\big[\hat{\mathcal T}_{N}-\mathcal T\big](V)\\
&\quad =
\Psi'\bigl(
\mathcal Q(V);
\big[\hat{\mathcal Q}_{N}-\mathcal Q\big](\hat V_N)
-
\big[\hat{\mathcal Q}_{N}-\mathcal Q\big](V)
\bigr)
+o_p(N^{-1/2}).
\end{aligned}
\]
By \eqref{eq:inf-soc:q-continuity}, the argument of
\(\Psi'(\mathcal Q(V);\cdot)\) is \(o_p(N^{-1/2})\) in
\(C(\mathcal X\times\mathcal U)\). Since
\(\Psi'(\mathcal Q(V);\cdot)\) is bounded,
we obtain \eqref{eq:inf-soc:continuity}.

(ii)
Let
\[
A_N
\coloneqq
\max_{x\in\mathcal X,\ u\in\mathcal U}
|\hat{\mathcal{Q}}_N(V)(x,u)-\mathcal{Q}(V)(x,u)|.
\]
Since \(\mathcal X\times\mathcal U\) is finite, 
the CLT implies
$
A_N=O_p(N^{-1/2})
$.
Using \eqref{eq:scaled-value-function-difference-bound}, 
we obtain 
$
N^{1/2}\|\hat{V}_N - V\|_\infty=O_p(1)
$.
By uniqueness and finiteness, the optimality gap
$\Delta$ in \eqref{eq:optimality-gap}
is strictly positive.  Define
\[
\mathcal E_N
\coloneqq
\left\{
2A_N+2\gamma\|\hat{V}_N - V\|_\infty<\Delta
\right\}.
\]
Then \(\mathrm{Prob}\{\mathcal E_N\}\to1\).

On \(\mathcal E_N\), the minimizers defining
$
(\hat{\mathcal T}_N\hat V_N)(x)
$,
$
(\hat{\mathcal T}_NV)(x)
$,
$(\mathcal T\hat V_N)(x)$,
and 
$(\mathcal TV)(x)$
are all equal to \(\pi^\star(x)\), for every \(x\in\mathcal X\).  
We verify this for \((\mathcal T\hat V_N)\).
The other verifications are similar. 
On \(\mathcal E_N\), for every \(x\in\mathcal X\),
\[
\begin{aligned}
[\mathcal{Q}(\hat V_N)](x,u)
-
[\mathcal{Q}(\hat V_N)](x,\pi^\star(x))
&
\ge
[\mathcal{Q}(V)](x,u)-[\mathcal{Q}(V)](x,\pi^\star(x))
-2\gamma\|\hat V_N-V\|_\infty
\\
&
\ge
\Delta-2\gamma\|\hat V_N-V\|_\infty
>0.
\end{aligned}
\]
Thus \(\pi^\star(x)\) is the minimizer in \((\mathcal T\hat V_N)(x)\).

Let us define 
$
Y_N \coloneqq N^{1/2}(\hat{V}_N-V)
$.
Hence,
on \(\mathcal E_N\),
\[
\begin{aligned}
&
N^{1/2}
\bigl(
[\hat{\mathcal T}_N-\mathcal T](\hat V_N)
-
[\hat{\mathcal T}_N-\mathcal T](V)
\big)
=
\gamma
\bigg[
\frac1N\sum_{i=1}^N
Y_N(F^\star(\cdot,\xi^{(i)}))
-
\mathbb{E}_{\xi \sim P}[
Y_N(F^\star(\cdot,\xi))
]
\bigg].
\end{aligned}
\]
Recalling $P^\star$ from 
\Cref{rem:finite-state-bellman}, 
we define its empirical counterpart by
\[
\hat P_N^\star(x,y)
\coloneqq
\frac1N\sum_{i=1}^N
\mathbf 1_{\{F^\star(x,\xi^{(i)})=y\}},
\quad x, y \in \mathcal{X}.
\]
Then, on \(\mathcal E_N\),
\[
\begin{aligned}
\left|
N^{1/2}
\bigl(
[\hat{\mathcal T}_N-\mathcal T](\hat V_N)
-
[\hat{\mathcal T}_N-\mathcal T](V)
\bigr)(x)
\right|                                      
& =
\gamma
\left|
\sum_{y\in\mathcal X}
Y_N(y)\bigl(\hat P_N^\star(x,y)-P^\star(x,y)\bigr)
\right|                                      \\
&\le
\gamma\|Y_N\|_\infty
\sum_{y\in\mathcal X}
|\hat P_N^\star(x,y)-P^\star(x,y)|.
\end{aligned}
\]
Since \(\mathcal X\) is finite, the CLT implies
\[
\max_{x\in\mathcal X}
\sum_{y\in\mathcal X}
|\hat P_N^\star(x,y)-P^\star(x,y)|
=
O_p(N^{-1/2}).
\]
Moreover, 
by \eqref{eq:scaled-value-function-difference-bound}, 
we have $\|Y_N\|_\infty=O_p(1)$.
Therefore, on \(\mathcal E_N\),
\[
N^{1/2}
\bigl(
[\hat{\mathcal T}_N-\mathcal T](\hat V_N)
-
[\hat{\mathcal T}_N-\mathcal T](V)
\bigr)
=o_p(1)
\quad\text{in} \quad  C(\mathcal X).
\]
Combined with \(\mathrm{Prob}\{\mathcal E_N\}\to1\), this gives
\eqref{eq:inf-soc:continuity}.
\end{proof}

\begin{proof}[{Proof of \Cref{lem:Lipschitz-equicontinuity}}]
We first note that the stage-cost term cancels in
\eqref{eq:inf-soc:q-continuity}. Indeed, 
\[
\begin{aligned}
&\bigl([\hat{\mathcal Q}_N-\mathcal Q](\hat V_N)
-
[\hat{\mathcal Q}_N-\mathcal Q](V)\bigr)(x,u)       \\
&\quad =
\gamma
\bigg(
\frac1N\sum_{i=1}^N
\bigl[\hat V_N-V\bigr](F(x,u,\xi^{(i)}))
-
\mathbb{E}_{\xi\sim P}
\Big[
\big[\hat V_N-V\big](F(x,u,\xi))
\Big]
\bigg).
\end{aligned}
\]
Therefore, for verifying
\eqref{eq:inf-soc:q-continuity}, it is enough to apply
Theorem~3.13.4 in \cite{Vaart2023} to the class
\[
\left\{
\xi\mapsto W(F(x,u,\xi)):
(x,u,W)\in\mathcal X\times\mathcal U\times\mathcal V
\right\}.
\]
The hypotheses ensure that this class is \(P\)-Donsker;
see Theorems~2.5.6 and 2.7.9,   Corollaries~2.7.2 and 2.7.15,
and Example~2.10.9 in \cite{Vaart2023}.
Using \eqref{eq:scaled-value-function-difference-bound}, 
we have $N^{1/2}\|\hat V_N-V\|_\infty=O_p(1)$.
Hence
\begin{align*}
    \sup_{(x, u) \in \mathcal{X} \times \mathcal{U}}
    \, 
    \mathbb{E}_{\xi \sim P}
    \big[
    \big(\hat V_N(F(x,u,\xi))-V(F(x,u,\xi))\big)^2
    \big]
     = o_p(1).
\end{align*}
Now an application of Theorem~3.13.4 in 
\cite{Vaart2023} implies the assertion.
\end{proof}

\subsection{Proofs of \Cref{thm:nonunique-policy-limit,lem:nonunique-compact-directional-linearization}}
\label{app:nonunique-policy-limit}

\begin{proof}[{Proof of \Cref{thm:nonunique-policy-limit}}]
\Cref{lem:statistical-limit} implies the CLT in 
\eqref{eq:Q-CLT}.
For \(h\in C(\mathcal X\times\mathcal U)\) and
\(W\in C(\mathcal X)\), we define
\[
[\Gamma(h,W)](x)
\coloneqq
\min_{u\in\mathcal U^\star(x)}
\left\{
h(x,u)+\gamma[\mathcal A W](x,u)
\right\}.
\]
The asserted nonlinear fixed point equation 
can be written as \(\mathfrak Z=\Gamma(\mathfrak Y,\mathfrak Z)\).
For fixed \(h\in C(\mathcal X\times\mathcal U)\), the map
\(W\mapsto\Gamma(h,W)\) is a \(\gamma\)-contraction on
\(C(\mathcal X)\). We denote its unique fixed point by \(S(h)\), that
is, \(S(h)=\Gamma(h,S(h))\). The solution map
\(S:C(\mathcal X\times\mathcal U)\to C(\mathcal X)\) is Lipschitz.

We define
\[
B_N\coloneqq
N^{1/2}(\hat{\mathcal Q}_N(V)-\mathcal Q(V)),
\qquad
Y_N\coloneqq N^{1/2}(\hat V_N-V).
\]
By \Cref{lem:statistical-limit},
\(B_N\rightsquigarrow\mathfrak Y\) in
\(C(\mathcal X\times\mathcal U)\). Moreover,
\[
\mathcal Q(\hat V_N)-\mathcal Q(V)
=
\gamma\mathcal A(\hat V_N-V).
\]
Thus the definition of \(\Delta_N\) and \eqref{eq:inf-soc:q-continuity} give
\[
\Delta_N=B_N+\gamma\mathcal A Y_N+o_p(1)
\quad
\text{in} \quad  C(\mathcal X\times\mathcal U).
\]

Using \(V=\Psi(\mathcal Q(V))\) and
\(\hat V_N=\Psi(\hat{\mathcal Q}_N(\hat V_N))\), we obtain
\[
Y_N
=
N^{1/2}
\Big[
\Psi\big(\mathcal Q(V)+N^{-1/2}\Delta_N\big)
-
\Psi(\mathcal Q(V))
\Big].
\]
By \eqref{eq:nonunique-directional-linearization},
\[
Y_N=\Psi'(\mathcal Q(V);\Delta_N)+o_p(1).
\]
Since \(h\mapsto\Psi'(\mathcal Q(V);h)\) is Lipschitz continuous,
\[
Y_N
=
\Psi'(\mathcal Q(V);B_N+\gamma\mathcal A Y_N)+o_p(1)
=
\Gamma(B_N,Y_N)+o_p(1).
\]

Since \(S(B_N)\) is the unique fixed point of
\(W\mapsto\Gamma(B_N,W)\),
\[
\begin{aligned}
\|Y_N-S(B_N)\|_\infty
&\le
\|\Gamma(B_N,Y_N)-\Gamma(B_N,S(B_N))\|_\infty
+o_p(1)                                      \\
&\le
\gamma\|Y_N-S(B_N)\|_\infty+o_p(1).
\end{aligned}
\]
Hence \(\|Y_N-S(B_N)\|_\infty=o_p(1)\). Since \(S\) is Lipschitz and
\(B_N\rightsquigarrow\mathfrak Y\), the continuous mapping theorem gives
\(S(B_N)\rightsquigarrow S(\mathfrak Y)\) in \(C(\mathcal X)\).
Therefore \(Y_N\rightsquigarrow S(\mathfrak Y)\), where
\(S(\mathfrak Y)\) is the unique fixed point of
\(W\mapsto\Gamma(\mathfrak Y,W)\).
\end{proof}

\begin{proof}[{Proof of
\Cref{lem:nonunique-compact-directional-linearization}}]
By \eqref{eq:scaled-value-function-difference-bound},
\(N^{1/2}(\hat V_N-V)=O_p(1)\) in \(C(\mathcal X)\). Since
\(\mathcal A\) is compact,
\(N^{1/2}\mathcal A(\hat V_N-V)\) is tight in
\(C(\mathcal X\times\mathcal U)\). Also,
\Cref{lem:statistical-limit} implies that
\(N^{1/2}(\hat{\mathcal Q}_N(V)-\mathcal Q(V))\) is tight. Hence, using
\eqref{eq:inf-soc:q-continuity} and
\(\mathcal Q(\hat V_N)-\mathcal Q(V)=\gamma\mathcal A(\hat V_N-V)\),
we obtain
\[
\Delta_N
=
N^{1/2}(\hat{\mathcal Q}_N(V)-\mathcal Q(V))
+
\gamma N^{1/2}\mathcal A(\hat V_N-V)
+
o_p(1),
\]
so \(\Delta_N\) is tight in \(C(\mathcal X\times\mathcal U)\).
By \Cref{lem:infimum-map-derivative}, \(\Psi\) is Hadamard
directionally differentiable at \(\mathcal Q(V)\). Combined with
tightness of $\Delta_N$,
the delta method, 
and a subsequence argument, we obtain 
\eqref{eq:nonunique-directional-linearization}.
\end{proof}

\section{Proofs for \Cref{sec:policy-saa}}
\label{app:variance-proofs}

\begin{proof}[Proof of \Cref{prop:policy-clt}]
By \Cref{ass:parametric-policy-class}, the objective in
\eqref{eq:parametric-policy-population} is continuous on \(\Theta\).
Since \(\Theta\) is compact, the population policy optimization problem
\eqref{eq:parametric-policy-population} has at least one solution.
By well specification, there exists \(\bar\theta\in\Theta\) such that
\(\pi_{\bar\theta}=\pi^\star\). Hence
\[
\mathbb E_{\boldsymbol{\xi} \sim P^\infty}
\bigl[
J(\bar\theta;\boldsymbol\xi)
\bigr]
=
V(x_0).
\]
Moreover, \(\Pi_\Theta\) consists of admissible stationary Markov policies
for \eqref{eq:inf-soc}. Therefore, for every \(\theta\in\Theta\),
\[
\mathbb E_{\boldsymbol{\xi} \sim P^\infty}
\bigl[
J(\theta;\boldsymbol\xi)
\bigr]
\geq
V(x_0).
\]
It follows that the optimal value of
\eqref{eq:parametric-policy-population} is \(V(x_0)\), and
\(\bar\theta\) is a solution.
Let \(\theta^\star\) be any solution of
\eqref{eq:parametric-policy-population}. By the assumption in
\Cref{prop:policy-clt},
\(\pi_{\theta^\star}=\pi^\star=\pi_{\bar\theta}\). Identifiability then
implies \(\theta^\star=\bar\theta\). Thus \(\bar\theta\) is the unique
solution of \eqref{eq:parametric-policy-population}.

To establish the CLT, we follow the proof of
Theorem~5.7 in \cite{SDR}. Although the theorem is stated for
finite-dimensional random vectors, its argument uses only i.i.d.\
observations, square integrability of the objective, and a
square-integrable Lipschitz modulus. Applied to the trajectory-valued
observations \(\boldsymbol\xi^{(1)},\ldots,\boldsymbol\xi^{(N)}\), it yields
\[
N^{1/2}
\bigl(
\hat\vartheta_N^{\mathrm{tr}}-V(x_0)
\bigr)
\rightsquigarrow
\mathcal N
\big(
0,
\operatorname{Var}_{\boldsymbol{\xi} \sim P^\infty}
\bigl(
J(\bar\theta;\boldsymbol\xi)
\bigr)
\big).
\]
Since \(\pi_{\bar\theta}=\pi^\star\), the limiting variance equals
\(\sigma_{\pi^\star,\gamma}^2(x_0)\) by
\eqref{eq:direct-variance}.
\end{proof}

Let us define
\begin{align*}
\mathsf{g}(x, \xi) \coloneqq f(x, \pi^\star(x),\xi) + \gamma V(F(x,\pi^\star(x),\xi)) 
- V(x).
\end{align*}

\begin{proof}[{Proof of \Cref{prop:random-path-variance-ratio}}]
By compactness and continuity, \(f\)
is uniformly bounded by, say, $M_f < \infty$.
Since \(V=\mathcal T V\), the contraction bound gives
$
\|V\|_\infty
\le
(1-\gamma)^{-1}M_f
$.
Thus \(\Phi(\cdot,\pi^\star(\cdot),\cdot)\) and \(\mathsf{g}\) are bounded.
Consequently, \(\mathcal{K}_\gamma\) is bounded, and the series below are
absolutely convergent.

We first prove the DP-SAA identity
\eqref{eq:DP-variance-discounted-occupation}.
Since
the operator norm of $\mathcal{L}$
is at most one, the 
von~Neumann series and \eqref{eq:inf-soc:limit} give
\[
\mathfrak G
=
\sum_{t\geq 0}
\gamma^t\mathcal L^t\mathfrak H
\quad\text{in} \quad  C(\mathcal X).
\]

Next, 
we prove $[\mathcal L^t h](x_0) = \mathbb{E}[h(\boldsymbol x_{t})]$ 
by induction on \(t\).  For \(t=0\), it reduces to
\(h(x_0)=h(\boldsymbol x_0)\), since \(\boldsymbol x_0=x_0\).  Suppose
it holds for some \(t\).  Then, by the definition of \(\mathcal L\) and
the Markov property of the closed-loop process,
\[
\begin{aligned}
[\mathcal L^{t+1}h](x_0)
=
[\mathcal L^t(\mathcal L h)](x_0)  
=
\mathbb{E}[[\mathcal Lh](\boldsymbol x_t)] 
=
\mathbb{E}\left[
\mathbb{E}\left[
h(\boldsymbol x_{t+1})
\,\middle|\,
\boldsymbol x_t
\right]
\right] 
=
\mathbb{E}[h(\boldsymbol x_{t+1})].
\end{aligned}
\]
Thus the identity holds for all \(t\ge0\).
Therefore
\[
\mathfrak G(x_0)
=
\sum_{t\geq 0}
\gamma^t
\mathbb{E}
\left[
\mathfrak H(\boldsymbol x_t)
\mid
\mathfrak H
\right].
\]

Using the independent copy
\((\tilde{\boldsymbol x}_t)_{t\ge0}\), Fubini's theorem, and the
covariance kernel of \(\mathfrak H\), we obtain
\[
\begin{aligned}
\operatorname{Var}(\mathfrak G(x_0))
&=
\sum_{s, t \geq 0}
\gamma^{s+t}
\operatorname{Cov}
\left(
\mathbb{E}[
\mathfrak H(\boldsymbol x_s)\mid\mathfrak H],
\mathbb{E}[
\mathfrak H(\boldsymbol x_t)\mid\mathfrak H]
\right)                                      \\
&=
\sum_{s, t \geq 0}
\gamma^{s+t}
\mathbb{E}
\left[
\operatorname{Cov}
\left(
\mathfrak H(\boldsymbol x_s),
\mathfrak H(\tilde{\boldsymbol x}_t)
\right)
\right]                                      
=
\sum_{s, t \geq 0}
\gamma^{s+t}
\mathbb{E}
\left[
\mathcal{K}_\gamma(\boldsymbol x_s,\tilde{\boldsymbol x}_t)
\right].
\end{aligned}
\]
This proves the first identity
in \eqref{eq:DP-variance-discounted-occupation}.
Since \(\boldsymbol x_s\) and
\(\tilde{\boldsymbol x}_t\) are independent with laws \(\mu_s\) and
\(\mu_t\),
\[
\mathbb{E}
\left[
\mathcal K_\gamma(\boldsymbol x_s,\tilde{\boldsymbol x}_t)
\right]
=
\int_{\mathcal X\times\mathcal X}
\mathcal K_\gamma(x,x')\,\mu_s(\mathrm dx)\mu_t(\mathrm dx').
\]
Hence
\[
\sigma_{\mathrm{DP},\gamma}^2(x_0)
=
\sum_{s, t \geq 0}
\gamma^{s+t}
\int_{\mathcal X\times\mathcal X}
\mathcal K_\gamma(x,x')\,\mu_s(\mathrm dx)\mu_t(\mathrm dx').
\]
On the other hand, by the definition of \(\nu_\gamma\),
\[
\nu_\gamma\otimes\nu_\gamma
=
(1-\gamma)^2
\sum_{s, t \geq 0}
\gamma^{s+t}\,\mu_s\otimes\mu_t ,
\]
where the equality is in the sense of finite measures. Therefore,
using boundedness of \(\mathcal K_\gamma\) to interchange the integral
and the absolutely convergent series,
\[
\begin{aligned}
&\int_{\mathcal X\times\mathcal X}
\mathcal K_\gamma(x,x')\,
\nu_\gamma(\mathrm dx)\nu_\gamma(\mathrm dx')
=
(1-\gamma)^2
\sum_{s, t \geq 0}
\gamma^{s+t}
\int_{\mathcal X\times\mathcal X}
\mathcal K_\gamma(x,x')\,\mu_s(\mathrm dx)\mu_t(\mathrm dx').
\end{aligned}
\]
Dividing by \((1-\gamma)^2\) proves
\eqref{eq:DP-variance-discounted-occupation}.

We now prove the direct SAA variance identity. 
The DP equation under the optimal policy implies
$
\mathbb{E}_{\xi \sim P}[\mathsf{g}(x,\xi)]
=
0
$, 
$x\in\mathcal X$.
For $t \geq 0$, let
$\xi_{[t]} \coloneqq  (\xi_0, \ldots, \xi_{t})$.
Then,
for $t\geq 1$, \(\boldsymbol x_t\) is \(\xi_{[t-1]}\)-measurable and \(\xi_t\)
is independent of \(\xi_{[t-1]}\). Hence
\[
\mathbb{E}[\mathsf{g}(\boldsymbol x_t,\xi_t)\mid\xi_{[t-1]}]=0.
\]
Thus \((\mathsf{g}(\boldsymbol x_t,\xi_t))_{t\ge0}\) is a martingale-difference sequence.
Moreover, for every \(T\ge0\),
\[
\sum_{t=0}^T\gamma^t \mathsf{g}(\boldsymbol x_t,\xi_t)
=
\sum_{t=0}^T
\gamma^t
f(\boldsymbol x_t,\pi^\star(\boldsymbol x_t),\xi_t)
-
V(x_0)
+
\gamma^{T+1}V(\boldsymbol x_{T+1}).
\]
Since \(V\) is bounded, the last term converges to zero in mean square.
Letting \(T\to\infty\), 
\[
\sum_{t\geq 0}
\gamma^t
f(\boldsymbol x_t,\pi^\star(\boldsymbol x_t),\xi_t)
-
V(x_0)
=
\sum_{t\geq 0}\gamma^t \mathsf{g}(\boldsymbol x_t,\xi_t)
\]
with convergence in mean square.
Since \((\mathsf{g}(\boldsymbol x_t,\xi_t))_{t\ge0}\) is a martingale-difference sequence,
\[
\operatorname{Var}\biggl(\sum_{t\geq 0}
\gamma^t
f(\boldsymbol x_t,\pi^\star(\boldsymbol x_t),\xi_t)\biggr)
=
\sum_{t\geq 0}
\gamma^{2t}
\mathbb{E}[\mathsf{g}(\boldsymbol x_t,\xi_t)^2].
\]
Using  \(\mathbb{E}_{\xi \sim P}[\mathsf{g}(x,\xi)]=0\), 
$
\mathbb{E}[\mathsf{g}(\boldsymbol x_t,\xi)^2\mid\xi_{[t-1]}]
=
\mathcal{K}_\gamma(\boldsymbol x_t,\boldsymbol x_t)
$.
Taking expectations gives
\[
\operatorname{Var}\biggl(\sum_{t\geq 0}
\gamma^t
f(\boldsymbol x_t,\pi^\star(\boldsymbol x_t),\xi_t)\biggr)
=
\sum_{t\geq 0}
\gamma^{2t}
\mathbb{E}
\left[
\mathcal{K}_\gamma(\boldsymbol x_t,\boldsymbol x_t)
\right].
\]
Since \(\boldsymbol x_t\) has law \(\mu_t\),
$
\mathbb{E}
\left[
\mathcal K_\gamma(\boldsymbol x_t,\boldsymbol x_t)
\right]
=
\int_{\mathcal X}
\mathcal K_\gamma(x,x)\,\mu_t(\mathrm dx)
$.
Thus, by boundedness of \(\mathcal K_\gamma\),
\[
\int_{\mathcal X}
\mathcal K_\gamma(x,x)\,\nu_{\gamma^2}(\mathrm dx)
=
(1-\gamma^2)
\sum_{t\geq 0}
\gamma^{2t}
\int_{\mathcal X}
\mathcal K_\gamma(x,x)\,\mu_t(\mathrm dx).
\]
Dividing by \(1-\gamma^2\) proves
\eqref{eq:direct-variance-discounted-occupation}. 
\end{proof}

\begin{proof}[Proof of \Cref{lem:param-finite-state-variances}]
We order the states as \((0,1,2)\) and identify \(C(\mathcal X)\) with
\(\mathbb R^3\) as in \Cref{rem:finite-state-bellman}. Then
\[
P^\star=
\begin{bmatrix}
0&1&0\\
0&0&1\\
0&1&0
\end{bmatrix},
\quad \text{and} \quad 
\operatorname{Cov}(\mathfrak H_\eta)=
\begin{bmatrix}
1&0&0\\
0&\eta^2&\eta^3\\
0&\eta^3&\eta^4
\end{bmatrix},
\]
where
\(\mathfrak H_\eta=(G_Y,\eta G_Z,\eta^2G_Z)^\top\) and \(G_Y,G_Z\) are
independent standard normal random variables.
Let \(e_1\coloneqq (1,0,0)^\top\). By the finite-state representation in
\Cref{rem:finite-state-bellman},
\[
e_1^\top(I-\gamma P^\star)^{-1}
=
\Big(
1,\frac{\gamma}{1-\gamma^2},
\frac{\gamma^2}{1-\gamma^2}
\Big).
\]
Consequently,
\[
e_1^\top(I-\gamma P^\star)^{-1}\mathfrak H_\eta
=
G_Y+
\frac{\gamma \eta(1+\gamma \eta)}{1-\gamma^2}G_Z,
\]
which gives
\[
\sigma_{\mathrm{DP},\eta,\gamma}^2(0)
=
1+
\frac{\gamma^2\eta^2(1+\gamma \eta)^2}{(1-\gamma^2)^2}.
\]

For the trajectory-based variance computation, we
use
\(\boldsymbol x_0=0\), \(\boldsymbol x_{2k+1}=1\), and
\(\boldsymbol x_{2k+2}=2\) for \(k\geq0\). Independence across time
therefore yields
\[
\begin{aligned}
\sigma_{\pi^\star,\eta,\gamma}^2(0)
=
1+
\sum_{k=0}^{\infty}\gamma^{4k+2}\eta^2
+
\sum_{k=0}^{\infty}\gamma^{4k+4}\eta^4 =
1+
\frac{\gamma^2\eta^2+\gamma^4\eta^4}{1-\gamma^4}.
\end{aligned}
\]
\end{proof}

\bibliography{CLT-Infinite-Horizon.bbl}

\end{document}